\documentclass[11pt]{article}

\usepackage[T1]{fontenc}
\usepackage{lmodern}

\usepackage{soul}
\usepackage{amsmath, amssymb, amsthm} 
\usepackage[margin=0.88in]{geometry}
\usepackage{color, graphicx}

\usepackage{etoolbox} 
\makeatletter
\patchcmd{\@maketitle}{\LARGE \@title}{\LARGE\bfseries\@title}{}{}

\renewcommand{\@seccntformat}[1]{\csname the#1\endcsname.\quad}
\makeatother

\definecolor{darkblue}{rgb}{0,0,.5}

\usepackage{hyperref}
\hypersetup{
	colorlinks=true,        
	linkcolor=darkblue,         
	citecolor=darkblue,         
	urlcolor=darkblue          
}

\usepackage[overload]{empheq}
\definecolor{myblue}{rgb}{.9, .9, 1}

\makeatletter
\def\th@plain{%
	\thm@notefont{}
	\itshape 
}
\def\th@definition{%
	\thm@notefont{}
	\normalfont 
}

\renewenvironment{proof}[1][\proofname]{\par
	\normalfont
	\topsep0\p@\@plus3\p@ \trivlist
	\item[\hskip\labelsep\itshape
	#1\@addpunct{.}]\ignorespaces
}{%
	\qed\endtrivlist
}
\makeatother

\newtheorem{theorem}{Theorem}[section]
\newtheorem{lemma}[theorem]{Lemma}
\newtheorem{corollary}[theorem]{Corollary}
\newtheorem{proposition}[theorem]{Proposition}
\newtheorem{fact}[theorem]{Fact}

\theoremstyle{definition}
\newtheorem{definition}[theorem]{Definition}
\theoremstyle{definition}

\theoremstyle{definition}
\newtheorem{remark}[theorem]{Remark}

\usepackage[shortlabels]{enumitem}
\setlist[enumerate]{nosep}

\allowdisplaybreaks

\newcommand{\menge}[2]{\{{#1}~\big |~{#2}\}} 
 
\newcommand{\Menge}[2]{\left\{{#1}~\Big|~{#2}\right\}}

\newcommand{\scal}[2]{\left\langle {#1},{#2} \right\rangle}

\newcommand{\NN}{\ensuremath{\mathbb N}}
\newcommand{\nnn}{\ensuremath{{n\in{\mathbb N}}}}

\newcommand{\RR}{\ensuremath{\mathbb R}}
\newcommand{\RX}{\ensuremath{\left]-\infty,+\infty\right]}}
\newcommand{\RP}{\ensuremath{\mathbb{R}_+}}
\newcommand{\RPP}{\ensuremath{\mathbb{R}_{++}}}

\newcommand{\argmin}{\ensuremath{\operatorname*{argmin}}}

\newcommand{\ran}{\ensuremath{\operatorname{ran}}}
\newcommand{\zer}{\ensuremath{\operatorname{zer}}}
\newcommand{\dom}{\ensuremath{\operatorname{dom}}}

\newcommand{\gra}{\ensuremath{\operatorname{gra}}}
\newcommand{\Fix}{\ensuremath{\operatorname{Fix}}}
\newcommand{\Id}{\ensuremath{\operatorname{Id}}}
\newcommand{\prox}{\ensuremath{\operatorname{Prox}}}

\begin{document}

\title{Adaptive Douglas--Rachford Splitting Algorithm\\
for the Sum of Two Operators}

\author{
Minh N.\ Dao\thanks{CARMA, University of Newcastle, Callaghan, NSW 2308, Australia. 
E-mail: \texttt{daonminh@gmail.com}}
~and~
Hung M.\ Phan\thanks{Department of Mathematical Sciences, Kennedy College of Sciences, University of Massachusetts Lowell, Lowell, MA 01854, USA.
E-mail: \texttt{hung\char`_phan@uml.edu}.}}

\date{October 9, 2019}

\maketitle

\begin{abstract}
The Douglas--Rachford algorithm is a classical and powerful splitting method for minimizing the sum of two convex functions and, more generally, finding a zero of the sum of two maximally monotone operators. Although this algorithm is well understood when the involved operators are monotone or strongly monotone, the convergence theory for weakly monotone settings is far from being complete. In this paper, we propose an adaptive Douglas--Rachford splitting algorithm for the sum of two operators, one of which is strongly monotone while the other one is weakly monotone. With appropriately chosen parameters, the algorithm converges globally to a fixed point from which we derive a solution of the problem. When one operator is Lipschitz continuous, we prove global linear convergence, which sharpens recent known results.
\end{abstract}

{\small

}

\section{Introduction}

Inclusion problems that involve finding a zero of the sum of two set-valued operators play an important role in various areas of variational analysis and optimization. For instance, under some constraint qualifications, the classical optimization problem of minimizing the sum of two convex functions can be converted to the problem of finding a zero of the sum of subdifferential operators of these functions.
One popular approach for the sum of two maximally monotone operators is to employ the Douglas--Rachford (DR) algorithm. This algorithm was originally introduced in 1956 by Douglas and Rachford \cite{DR56} to numerically solve a system of linear equations arising in heat conduction. In 1979, Lions and Mercier made the algorithm applicable to a broad class of optimization problems through the seminal work \cite{LM79}. More specifically, they proved that each sequence generated by the DR algorithm converges weakly to a fixed point which is then used to derive a solution of the original problem. This result was later strengthened by Svaiter \cite{Sva11} in which weak convergence of the \emph{shadow} sequence to a solution was shown.
In the formulation of the DR algorithm, each step involves computing the resolvent of a single operator, and hence, it is often referred to as a {\em splitting} algorithm. Since mathematical structures emerging from applications are usually complex and difficult to analyze as a whole object, the idea of splitting is extremely important as it helps the calculation on simple components that make up the entire mathematical model. It is worth mentioning (see, e.g., \cite{EB92}) that several splitting methods such as the \emph{method of partial inverses} \cite{Spi83} and the \emph{alternating direction method of multipliers} (ADMM) \cite{GM76} can be written in the form of the DR algorithm, which itself can be transformed into the \emph{proximal point algorithm} \cite{Roc76}. Other splitting schemes can be found in \cite{BCH13,BriCom11,Com09} and the references therein. 

When applied to two normal cone operators, the DR algorithm can be used to solve the \emph{feasibility problem} of finding a common point of two sets. In this context, the DR algorithm possesses many good properties; for example, it finds a best approximation point when the intersection of sets is empty \cite{BCL04,BDM16,BM17}, it finds an exact solution after only a finite number of iterations under verifiable conditions \cite{ABT16,BD17,BDNP16b}, and it converges globally in some nonconvex settings \cite{Ben15,DT18} while it converges locally with linear or sublinear rate under some regularity assumptions \cite{BLT17, HL13, Phan16}. In the absence of constraint qualifications, \cite{BDNP16a} suggests that the DR algorithm outperforms the well-known \emph{method of alternating projections}.
In attempting to generalize the DR algorithm for feasibility problems, several parameters were added to its formulation \cite{BST15, DP18mor, DP18jogo, FG17}. In this case, one has the freedom to modify the parameters that are associated with the projections without giving up the solution. This approach is possible because the underlying normal cone operators have homogeneous values, which allows for scaling them independently. The situation changes completely when working with general problems where two involved operators may no longer have such homogeneity. In this case, a naive scaling may destroy the ability to solve the original problem. Therefore, we aim to overcome this hurdle by proposing an adaptive approach.

The paper is devoted to the convergence analysis of the {\em adaptive DR algorithm} for finding a zero of the sum of $\alpha$- and $\beta$-monotone operators,
in which $\alpha$-monotonicity is a unification of strong and weak monotonicity (see Definition~\ref{d:alpha_mono}). This situation arises in various important applications; see \cite{GHY17} for a brief discussion. The main contributions are summarized below.

{\bf (R1)} We incorporate parameters into the DR algorithm so that the weak convergence to some fixed point is achieved (see Theorem~\ref{t:DRcvg}). The chosen parameters then allow us to derive a solution to the original problem by using the shadow of the fixed point. In addition, the shadow sequences converge strongly to the solution whenever the strong monotonicity strictly outweighs the weak counterpart. We show by a simple proof that the rate of asymptotic regularity of the adaptive DR operator is $o(1/\sqrt{n})$. As expected, these results are also valid for the classical DR algorithm.

{\bf (R2)} Under Lipschitz continuity assumption, we prove that the convergence is strong with linear rate (see Theorems~\ref{t:linear} and \ref{t:linearB}) and that our linear rate refines previous results (see Corollary~\ref{c:classicalDR} and Remark~\ref{r:L_improved}). We note a particular result in Theorem~\ref{t:linear}\ref{t:linear_reduce} that when one operator is Lipschitz continuous and the other operator is strongly monotone, the adaptive DR algorithm converges linearly as long as the strong monotonicity constant is greater than the Lipschitz constant. This is interesting since no monotonicity assumption is imposed on the Lipschitz operator!

To the best of our knowledge, the results are {\em new} and encompass several contemporary works in this direction. Indeed, our results provide a consolidation for the classical DR algorithm and its adaptive version. In particular, we show how the parameters play a role in the convergence analysis of the algorithm.

The remainder of the paper is organized as follows. Section~\ref{s:prelim} supplies definitions and facts that are  necessary for our analysis. In Section~\ref{s:relaxed_resol}, we define and study various relevant properties of $\alpha$-monotone operators with and without Lipschitz assumptions. The main results for the adaptive DR algorithm and its convergence analysis are presented in Section~\ref{s:aDR}. Section~\ref{s:applications} contains some applications to structured minimization problems. Finally, concluding remarks and comments are given in Section~\ref{s:concl}.

\section{Preliminaries}
\label{s:prelim}

Throughout this work, $X$ is a real Hilbert space with inner product $\scal{\cdot}{\cdot}$ and induced norm $\|\cdot\|$. 
The set of nonnegative integers is denoted by $\NN$, the set of real numbers by $\RR$, the set of nonnegative real numbers by $\RP := \menge{x \in \RR}{x \geq 0}$, and the set of the positive real numbers by $\RPP := \menge{x \in \RR}{x >0}$.
We use the notation $A\colon X\rightrightarrows X$ to indicate that $A$ is a set-valued operator on $X$ and the notation $A\colon X\to X$ to indicate that $A$ is a single-valued operator on $X$. 

Let $A$ be an operator on $X$. The \emph{domain} of $A$ is $\dom A :=\menge{x\in X}{Ax\neq \varnothing}$, the \emph{graph} of $A$ is $\gra A :=\menge{(x,u)\in X\times X}{u\in Ax}$, and the set of \emph{fixed points} of $A$ is $\Fix A :=\menge{x\in X}{x\in Ax}$. The \emph{inverse} of $A$, denoted by $A^{-1}$, is the operator with graph $\gra A^{-1} :=\menge{(u,x)\in X\times X}{u\in Ax}$.
We say that $A$ is \emph{Lipschitz continuous} with constant $\ell\in \RP$ if it is single-valued and
\begin{equation}
\forall x, y\in \dom A,\quad \|Ax-Ay\|\leq \ell\|x-y\|.
\end{equation}
The operator $A$ is \emph{nonexpansive} if it is Lipschitz continuous with constant $1$, i.e.,
\begin{equation}
\forall x, y\in \dom A,\quad \|Ax-Ay\|\leq \|x-y\|.
\end{equation}
An operator $A\colon X\rightrightarrows X$ is said to be \emph{monotone} if
\begin{equation}
\forall (x, u), (y, v)\in \gra A,\quad \scal{x-y}{u-v}\geq 0,
\end{equation}
and said to be \emph{maximally monotone} if it is monotone and there exists no monotone operator $B\colon X\rightrightarrows X$ such that $\gra B$ properly contains $\gra A$. The \emph{resolvent} of $A\colon X\rightrightarrows X$ is defined by
\begin{equation}
J_A:= (\Id+ A)^{-1},
\end{equation}
where $\Id$ is the identity operator. The \emph{relaxed resolvent} of $A$ with parameter $\lambda\in\RP$ is defined by
\begin{equation}
J_A^\lambda:= (1-\lambda)\Id+ \lambda J_A.
\end{equation}
Next, we recall an important characterization of maximally monotone operators.
\begin{fact}
\label{f:Minty}
Let $A\colon X\rightrightarrows X$ be monotone and let $\gamma\in \RPP$. Then $\dom J_{\gamma A} =X$ if and only if $A$ is maximally monotone.
\end{fact}
\begin{proof}
By definition, $\dom J_{\gamma A} =\ran(\Id+\gamma A) :=(\Id+\gamma A)(X)$. Since $\gamma\in \RPP$, it holds that $\gamma A$ is monotone. According to Minty's theorem (see, e.g., \cite[Theorem~21.1]{BC17}), $\dom J_{\gamma A} =\ran(\Id+\gamma A) =X$ if and only if $\gamma A$ is maximally monotone. By \cite[Proposition~20.22]{BC17}, the latter occurs if and only if $A$ is maximally monotone. 
\end{proof}

We conclude this section with the following useful identity whose omitted proof is straightforward. For all $s, t\in X$ and all $\sigma, \tau\in \RR$,
\begin{equation}
\label{e:identity}
\|\sigma s+\tau t\|^2 =\sigma(\sigma+\tau)\|s\|^2+\tau(\sigma+\tau)\|t\|^2 -\sigma\tau\|s-t\|^2,
\end{equation}
which is equivalent to 
\begin{equation}
\label{e:identity'}
\sigma\|s\|^2+\tau\|t\|^2 =\frac{\sigma\tau}{\sigma+\tau}\|s-t\|^2 +\frac{1}{\sigma+\tau}\|\sigma s+\tau t\|^2
\end{equation}
whenever $\sigma+\tau\neq 0$.

\section{Relaxed resolvents of $\alpha$-monotone operators}
\label{s:relaxed_resol}

\begin{definition}[$\alpha$-monotonicity]
\label{d:alpha_mono}
An operator $A\colon X\rightrightarrows X$ is said to be \emph{$\alpha$-monotone} ($\alpha\in \RR$) if
\begin{equation}
\label{e:alpha-mono}
\forall (x,u), (y,v)\in \gra A,\quad \scal{x-y}{u-v}\geq \alpha\|x-y\|^2.
\end{equation}
The constant $\alpha$ is referred to as the \emph{monotonicity constant}.
We also say that $A$ is \emph{maximally $\alpha$-monotone} if it is $\alpha$-monotone and there is no $\alpha$-monotone operator whose graph strictly contains $\gra A$.  
\end{definition}

We note that $0$-monotonicity simply means monotonicity, that if $\alpha>0$, then $\alpha$-monotonicity is precisely the notion of \emph{strong monotonicity} \cite[Definition~22.1(iv)]{BC17}, and that if $\alpha<0$, then $\alpha$-monotonicity can be referred to as \emph{weak monotonicity}.
For detailed discussions on maximal monotonicity and its variants as well as the connection to optimization problems, we refer the reader to \cite{BC17, Bor10, BI08}.

\begin{lemma}[monotonicity versus $\alpha$-monotonicity]
\label{l:transfer}
Let $A\colon X\rightrightarrows X$ and let $\alpha, \beta\in \RR$. Then the following hold:
\begin{enumerate}
\item\label{l:transfer_mono} 
$A$ is $\alpha$-monotone if and only if $A-\beta\Id$ is $(\alpha-\beta)$-monotone.
\item\label{l:transfer_maxmono}
$A$ is maximally $\alpha$-monotone if and only if $A-\beta\Id$ is maximally $(\alpha-\beta)$-monotone. 
\end{enumerate}
Consequently, $A$ is (resp., maximally) $\alpha$-monotone if and only if $A-\alpha\Id$ is (resp., maximally) monotone.
\end{lemma}
\begin{proof}
\ref{l:transfer_mono}: We first have the equivalences
\begin{subequations}
\begin{align}
(x,u)\in \gra A &\iff (x,u-\beta x)\in \gra(A-\beta\Id), \\
(y,v)\in \gra A &\iff (y,v-\beta y)\in \gra(A-\beta\Id),
\end{align}
\end{subequations}
and
\begin{equation}
\scal{x-y}{u-v}\geq \alpha\|x-y\|^2 \iff \scal{x-y}{(u-\beta x)-(v-\beta y)}\geq (\alpha-\beta)\|x-y\|^2,
\end{equation}
from which the conclusion follows.

\ref{l:transfer_maxmono}: Assume that $A$ is maximally $\alpha$-monotone. By \ref{l:transfer_mono}, $A-\beta\Id$ is $(\alpha-\beta)$-monotone. Now, suppose that $A-\beta\Id$ is not maximally $(\alpha-\beta)$-monotone. Then there must exist $B'\colon X\rightrightarrows X$ such that $B'$ is $(\alpha-\beta)$-monotone and $\gra(A-\beta\Id)\subsetneq \gra B'$. It follows that $B :=B'+\beta\Id$ is $\alpha$-monotone due to \ref{l:transfer_mono} and that $\gra A\subsetneq \gra B$, which contradict the maximal $\alpha$-monotonicity of $A$. We deduce that if $A$ is maximally $\alpha$-monotone, then $A-\beta\Id$ is maximally $(\alpha-\beta)$-monotone. This also implies that if $A-\beta\Id$ is maximally $(\alpha-\beta)$-monotone, then $A =(A-\beta\Id)+\beta\Id$ is maximally $\alpha$-monotone, and we are done. 
\end{proof}

\begin{lemma}[resolvents of $\alpha$-monotone operators]
\label{l:resol}
Let $A\colon X\rightrightarrows X$ be $\alpha$-monotone and let $\gamma\in \RPP$. Then the following hold:
\begin{enumerate}
\item\label{l:resol_gen}
For all $(x,a), (y,b)\in \gra J_{\gamma A}$,
\begin{subequations}
\begin{align}
\label{e:resol_gen}
\scal{x-y}{a-b} &\geq (1+\gamma\alpha)\|a-b\|^2 \quad\text{and}\\
\|x-y\| &\geq (1+\gamma\alpha)\|a-b\|.
\end{align} 
\end{subequations}
\item\label{l:resol_single}
If $J_{\gamma A}$ is single-valued, then, for all $x,y\in \dom J_{\gamma A}$,
\begin{equation}
\scal{x-y}{J_{\gamma A}x-J_{\gamma A}y}\geq (1+\gamma\alpha)\|J_{\gamma A}x-J_{\gamma A}y\|^2,
\end{equation}
i.e., $J_{\gamma A}$ is $(1+\gamma\alpha)$-cocoercive.
\end{enumerate}
\end{lemma}
\begin{proof}
\ref{l:resol_gen}: Let $(x,a), (y,b)\in \gra J_{\gamma A}$. Then $x\in (\Id+\gamma A)a$, $y\in (\Id+\gamma A)b$, and so $x =a+\gamma u$, $y =b+\gamma v$ for some $u\in Aa$, $v\in Ab$. We derive from the $\alpha$-monotonicity of $A$ that 
\begin{subequations}
\begin{align}
\scal{x-y}{a-b}
&=\scal{(a+\gamma u)-(b+\gamma v)}{a-b}\\
&=\|a-b\|^2+\gamma\scal{a-b}{u-v}\\
&\geq \|a-b\|^2+\gamma\alpha\|a-b\|^2\\
&=(1+\gamma\alpha)\|a-b\|^2.
\end{align}
\end{subequations}
Now, by the Cauchy--Schwarz inequality,
\begin{equation}
\|x-y\|\|a-b\|\geq \scal{x-y}{a-b}\geq (1+\gamma\alpha)\|a-b\|^2.
\end{equation}
which gives $\|x-y\|\geq (1+\gamma\alpha)\|a-b\|$ while noting that this is trivial when $a =b$. 

\ref{l:resol_single}: This is a direct consequence of \ref{l:resol_gen}.
\end{proof}

\begin{proposition}[single-valuedness and full domain]
\label{p:resol}
Let $A\colon X\rightrightarrows X$ be $\alpha$-monotone and let $\gamma\in \RPP$ such that $1+\gamma\alpha >0$. Then the following hold:
\begin{enumerate}
\item\label{p:resol_single} 
$J_{\gamma A}$ is single-valued.
\item\label{p:resol_dom} 
$\dom J_{\gamma A} =X$ if and only if $A$ is maximally $\alpha$-monotone.
\end{enumerate}
\end{proposition}
\begin{proof}
\ref{p:resol_single}: This follows from Lemma~\ref{l:resol}\ref{l:resol_gen}.

\ref{p:resol_dom}: By Lemma~\ref{l:transfer}\ref{l:transfer_mono}, $A':=A-\alpha\Id$ is monotone.
Noting that $(\beta T)^{-1} =T^{-1}\circ \frac{1}{\beta}\Id$ for any operator $T$ and any $\beta\in \RR\smallsetminus\{0\}$, we have
\begin{subequations}
\begin{align}
J_{\gamma A}&=(\Id+\gamma A)^{-1}=\big((1+\gamma\alpha)\Id+\gamma(A-\alpha\Id)\big)^{-1}\\
&=\left(\Id+\frac{\gamma}{1+\gamma\alpha} A'\right)^{-1} \circ \left(\frac{1}{1+\gamma\alpha}\Id\right)\\
&=J_{\frac{\gamma}{1+\gamma\alpha}A'}\circ \left(\frac{1}{1+\gamma\alpha}\Id\right).
\end{align}
\end{subequations}
It follows that 
\begin{subequations}
\begin{align}
\dom J_{\gamma A} =X & \iff \dom J_{\frac{\gamma}{1+\gamma\alpha}A'} =X &&\\
& \iff \text{$A'$ is maximally monotone} &&\hspace*{-.8in}\text{(by Fact~\ref{f:Minty})} \\
& \iff \text{$A$ is maximally $\alpha$-monotone} &&\hspace*{-.8in}\text{(by Lemma~\ref{l:transfer}\ref{l:transfer_maxmono})}.
\end{align}
\end{subequations}
The proof is complete.
\end{proof}

Next, we further characterize the maximal $\alpha$-monotonicity.

\begin{proposition}[maximal $\alpha$-monotonicity]
\label{p:maxalpha}
The following statements hold: 
\begin{enumerate}
\item\label{p:maxalpha_mono} 
Let $A\colon X\rightrightarrows X$ and $\alpha\in \RP$. Then $A$ is maximally $\alpha$-monotone if and only if $A$ is $\alpha$-monotone and maximally monotone.
\item\label{p:maxalpha_cont} 
Let $A\colon X\to X$ and $\alpha\in \RR$. Then $A$ is maximally $\alpha$-monotone if $A$ is $\alpha$-monotone and continuous with full domain.
\end{enumerate}
\end{proposition}
\begin{proof}
\ref{p:maxalpha_mono}: Since $\alpha\geq 0$, it follows from Fact~\ref{f:Minty} that $A$ is $\alpha$-monotone and maximally monotone if and only if $A$ is $\alpha$-monotone and $\dom J_A =X$, which, by Proposition~\ref{p:resol}\ref{p:resol_dom}, happen if and only if $A$ is maximally $\alpha$-monotone.

\ref{p:maxalpha_cont}: Set $A' :=A-\alpha\Id$. Then $A'$ is monotone (due to Lemma~\ref{l:transfer}\ref{l:transfer_mono}) and continuous with full domain. By \cite[Corollary~20.28]{BC17}, $A'$ is maximally monotone, and by Lemma~\ref{l:transfer}\ref{l:transfer_maxmono}, $A$ is maximally $\alpha$-monotone. 
\end{proof}

In fact, an anonymous colleague has led us to the simple but important equivalence in Proposition~\ref{p:maxalpha}\ref{p:maxalpha_mono}.
From now on, we will simply use maximal $\alpha$-monotonicity whenever convenient.

\begin{remark}[Lipschitz $\alpha$-monotone operators]
\label{r:Lip_mono}
Suppose that $A$ is Lipschitz continuous with constant $\ell$. Then $A$ is single-valued and 
\begin{equation}
\forall x,y\in\dom A,\quad
|\scal{x-y}{Ax-Ay}|\leq \|x-y\|\cdot\|Ax-Ay\|\leq \ell\|x-y\|^2,
\end{equation}
which yields
\begin{equation}
\label{e:L-inequality}
\forall x,y\in\dom A,\quad -\ell\|x-y\|^2\leq \scal{x-y}{Ax-Ay}\leq \ell\|x-y\|^2.
\end{equation}
We immediately deduce that $A$ is $(-\ell)$-monotone. Now suppose, in addition, that $A$ is $\alpha$-monotone. On the one hand, we can always assume without loss of generality that $\alpha\geq-\ell$. On the other hand, it follows from the $\alpha$-monotonicity and \eqref{e:L-inequality} that $\alpha\leq \ell$ as soon as $\dom A$ has more than one element. Therefore, unless otherwise stated, whenever $A$ is both $\alpha$-monotone and Lipschitz continuous with constant $\ell$, we assume that $|\alpha|\leq\ell$.
\end{remark}

As seen in the following lemma, when an $\alpha$-monotone operator is also Lipschitz continuous, its resolvent possesses metric properties \emph{stronger} than Lemma~\ref{l:resol}. Some of these properties were also observed in \cite{Gis17,MV18} for the $\alpha\geq 0$ case.

\begin{lemma}[resolvents of Lipschitz $\alpha$-monotone operators]
\label{l:L-resol}
Let $A\colon X\to X$ be Lipschitz continuous with constant $\ell$ and let $\gamma\in \RPP$. Then the following hold:
\begin{enumerate}
\item\label{l:L-resol_only} 
For all $(x,a), (y,b)\in \gra J_{\gamma A}$,
\begin{subequations}
\begin{align}
\label{e:expan}
\|a-b\|&\geq \frac{1}{1+\gamma\ell}\|x-y\|,\\
\label{e:L-resol_only}
\scal{x-y}{a-b}&\geq \frac{1}{2}\|x-y\|^2 +\frac{1}{2}(1-\gamma^2\ell^2)\|a-b\|^2,
\end{align}
\end{subequations}
and if $\gamma\ell\leq 1$, then
\begin{equation}\label{e:L-resol_only1}
\scal{x-y}{a-b}\geq \frac{1}{1+\gamma\ell}\|x-y\|^2.
\end{equation}
\item\label{l:L-resol_mono}
If $A$ is $\alpha$-monotone with $1+\gamma\alpha >0$, then, for all $x, y\in \dom J_{\gamma A}$, 
\begin{equation}
\scal{x-y}{J_{\gamma A}x-J_{\gamma A}y}\geq (1+\gamma\alpha)\alpha_J\|x-y\|^2,
\end{equation}
where 
\begin{subequations}
\label{e:alpha_J}
\begin{equation}
\alpha_J 
:=\begin{cases}
\dfrac{1}{1+2\gamma\alpha+\gamma^2\ell^2} &\text{if~} \gamma\ell\geq 1 
,\\
\dfrac{1}{(1+\gamma\alpha)(1+\gamma\ell)} &\text{if~} \gamma\ell\leq 1;
\end{cases}
\end{equation}
and if additionally $A$ satisfies \eqref{e:alpha-mono} with equality, then 
\begin{equation}
\alpha_J :=\frac{1}{1+2\gamma\alpha+\gamma^2\ell^2}. 
\end{equation}
\end{subequations}
\end{enumerate}
\end{lemma}
\begin{proof}
\ref{l:L-resol_only}: Let $(x,a), (y,b)\in \gra J_{\gamma A}$. Then $x =a+\gamma Aa$ and $y =b+\gamma Ab$. By Lipschitz continuity, $\|Aa-Ab\|\leq \ell\|a-b\|$. It follows that
\begin{equation}
\|x-y\|\leq \|a-b\| +\gamma\|Aa-Ab\|\leq (1+\gamma\ell)\|a-b\|
\end{equation}
and that
\begin{subequations}
\label{e:L-resol}
\begin{align}
2\scal{x-y}{a-b} &=\|x-y\|^2 +\|a-b\|^2 -\|(x-a)-(y-b)\|^2\\
&=\|x-y\|^2 +\|a-b\|^2 -\gamma^2\|Aa-Ab\|^2\\
&\geq \|x-y\|^2 +(1-\gamma^2\ell^2)\|a-b\|^2.
\end{align}
\end{subequations}
If $\gamma\ell\leq 1$, then combining the above inequalities yields
\begin{equation}
2\scal{x-y}{a-b}\geq \|x-y\|^2 +\frac{1-\gamma^2\ell^2}{(1+\gamma\ell)^2}\|x-y\|^2 =\frac{2}{1+\gamma\ell}\|x-y\|^2,
\end{equation}
and we get the claim.

\ref{l:L-resol_mono}: We first note that $J_{\gamma A}$ is single-valued due to Proposition~\ref{p:resol}\ref{p:resol_single}. Then \eqref{e:L-resol} reads as
\begin{equation}
\label{e:L-resol'}
2\scal{x-y}{J_{\gamma A}x-J_{\gamma A}y}\geq \|x-y\|^2 +(1-\gamma^2\ell^2)\|J_{\gamma A}x-J_{\gamma A}y\|^2.
\end{equation}
We claim that if $\gamma\ell\geq 1$ or $A$ satisfies \eqref{e:alpha-mono} with equality, then 
\begin{equation}
\label{e:L-resol''}
2\scal{x-y}{J_{\gamma A}x-J_{\gamma A}y}\geq \|x-y\|^2 +\frac{1-\gamma^2\ell^2}{1+\gamma\alpha}\scal{x-y}{J_{\gamma A}x-J_{\gamma A}y}.
\end{equation} 
Indeed, the former case implies $1-\gamma^2\ell^2\leq 0$ and, by combining \eqref{e:L-resol'} with Lemma~\ref{l:resol}\ref{l:resol_single} and noting that $1+\gamma\alpha >0$, we get \eqref{e:L-resol''}.  
In the latter case, Lemma~\ref{l:resol}\ref{l:resol_single} reduces to
\begin{equation}
\scal{x-y}{J_{\gamma A}x-J_{\gamma A}y} =(1+\gamma\alpha)\|J_{\gamma A}x-J_{\gamma A}y\|^2.
\end{equation} 
Substituting this into \eqref{e:L-resol'}, we also obtain \eqref{e:L-resol''}. 

Now, in view of Remark~\ref{r:Lip_mono}, $1+2\gamma\alpha+\gamma^2\ell^2\geq 1+2\gamma\alpha+\gamma^2\alpha^2 =(1+\gamma\alpha)^2 >0$. It thus follows from \eqref{e:L-resol''} that 
\begin{equation}
\scal{x-y}{J_{\gamma A}x-J_{\gamma A}y}\geq \frac{1+\gamma\alpha}{1+2\gamma\alpha+\gamma^2\ell^2}\|x-y\|^2.
\end{equation}
Finally, if $\gamma\ell\leq 1$, then, by \ref{l:L-resol_only}, 
\begin{equation}
\scal{x-y}{J_{\gamma A}x-J_{\gamma A}y}\geq \frac{1}{1+\gamma\ell}\|x-y\|^2,
\end{equation}
and the conclusion follows.
\end{proof}

\begin{remark}[a case of equality in \eqref{e:alpha-mono}]
\label{r:equality}
At first glance, an operator that satisfies \eqref{e:alpha-mono} with equality seems unusual. Nevertheless, it turns out that there is a special operator class that falls into this case. Indeed, let $S\colon X\to X$ be a linear skew operator, i.e., $S^* =-S$. Define $A :=S+\alpha\Id$ with $\alpha\in \RR$. Then, for all $x,y\in \dom A =\dom S$, we have that 
\begin{equation}
\scal{x-y}{Ax-Ay} =\scal{x-y}{S(x-y)} +\alpha\|x-y\|^2 =\alpha\|x-y\|^2,
\end{equation}
i.e., $A$  satisfies \eqref{e:alpha-mono} with equality.
\end{remark}

Next, we turn our attention to the relaxed resolvent of an $\alpha$-monotone operator, which is a special case of the linear combination of the resolvent and the identity. We will establish two types of metric estimations for relaxed resolvents, one for general $\alpha$-monotone operators and one for Lipschitz $\alpha$-monotone operators. In fact, the latter case possesses some Lipschitz estimations, which help when proving the linear convergence in the next section.

\begin{proposition}[linear combinations of resolvents and the identity]
\label{p:cresol}
Let $A\colon X\rightrightarrows X$ be $\alpha$-monotone and let $\gamma\in \RPP$. Set $J:=J_{\gamma A}$ and define $Q :=\nu\Id+\lambda J$ with $\nu, \lambda\in \RR$. 
\begin{enumerate}
\item\label{p:cresol_mono} 
Suppose that $J$ is single-valued and $\nu\lambda\leq 0$. Then, for all $x,y\in \dom J$, 
\begin{equation}
\|Qx-Qy\|^2\leq \nu^2\|x-y\|^2 +\lambda\big(2\nu(1+\gamma\alpha)+\lambda\big)\|Jx-Jy\|^2.
\end{equation}
\item\label{p:cresol_Lip}
Suppose that $A$ is Lipschitz continuous with constant $\ell$, $1+\gamma\alpha >0$, and $\lambda\big(2\nu(1+\gamma\alpha)+\lambda\big)\leq 0$. Then $Q$ is Lipschitz continuous with constant    
\begin{equation}
\label{e:LipQ}
\rho :=\sqrt{\nu^2+\lambda\big(2\nu(1+\gamma\alpha)+\lambda\big)\alpha_J},
\end{equation}
where $\alpha_J$ is defined as \eqref{e:alpha_J}.
If additionally $\lambda\left(\nu+\frac{\lambda}{1-\gamma^2\ell^2}\right)\geq 0$ whenever $\gamma\ell <1$, then the Lipschitz constant \eqref{e:LipQ} can be improved to
\begin{equation}
\overline{\rho} :=\sqrt{\nu^2+\frac{\lambda\big(2\nu(1+\gamma\alpha)+\lambda\big)}{1+2\gamma\alpha+\gamma^2\ell^2}}\leq \rho.
\end{equation}
\end{enumerate}
\end{proposition}
\begin{proof}
Let $x,y\in \dom J$. By the definition of $Q$,
\begin{subequations}
\label{e:identityQ}
\begin{align}
\|Qx-Qy\|^2
&=\|\nu(x-y)+\lambda(Jx-Jy)\|^2\\
&=\nu^2\|x-y\|^2 +2\nu\lambda\scal{x-y}{Jx-Jy} +\lambda^2\|Jx-Jy\|^2.
\end{align}
\end{subequations}

\ref{p:cresol_mono}: Since $\nu\lambda\leq 0$, combining \eqref{e:identityQ} with Lemma~\ref{l:resol}\ref{l:resol_single} yields
\begin{subequations}
\begin{align}
\|Qx-Qy\|^2
&\leq \nu^2\|x-y\|^2 +2\nu\lambda(1+\gamma\alpha)\|Jx-Jy\|^2 +\lambda^2\|Jx-Jy\|^2\\
&=\nu^2\|x-y\|^2 +\lambda\big(2\nu(1+\gamma\alpha)+\lambda\big)\|Jx-Jy\|^2.
\end{align}
\end{subequations}

\ref{p:cresol_Lip}: First, according to Proposition~\ref{p:resol}\ref{p:resol_single}, $J$ is single-valued, and so is $Q$. Next, using \eqref{e:identityQ}, Lemma~\ref{l:resol}\ref{l:resol_single}, and Lemma~\ref{l:L-resol}\ref{l:L-resol_mono} and noting that $\lambda\big(2\nu(1+\gamma\alpha)+\lambda\big)\leq 0$, we have
\begin{subequations}
\label{e:cresol_Lip}
\begin{align}
\|Qx-Qy\|^2 
&\leq \nu^2\|x-y\|^2 +2\nu\lambda\scal{x-y}{Jx-Jy} +\frac{\lambda^2}{1+\gamma\alpha}\scal{x-y}{Jx-Jy}\\ 
&= \nu^2\|x-y\|^2 +\frac{\lambda\big(2\nu(1+\gamma\alpha)+\lambda\big)}{1+\gamma\alpha}\scal{x-y}{Jx-Jy}\\  
&\leq \nu^2\|x-y\|^2 +\lambda\big(2\nu(1+\gamma\alpha)+\lambda\big)\alpha_J\|x-y\|^2 =\rho^2\|x-y\|^2, 
\end{align}
\end{subequations}
which implies that $Q$ is Lipschitz continuous with constant $\rho$.

For the last statement, we show that $\alpha_J$ in formula \eqref{e:LipQ} can be replaced by 
\begin{equation}
\overline{\alpha}_J :=\frac{1}{1+2\gamma\alpha+\gamma^2\ell^2}.
\end{equation}
If $\scal{x-y}{Jx-Jy}\geq (1+\gamma\alpha)\overline{\alpha}_J\|x-y\|^2$, then \eqref{e:cresol_Lip} also holds with $\alpha_J$ replaced by $\overline{\alpha}_J$. Now, assume that $\scal{x-y}{Jx-Jy} <(1+\gamma\alpha)\overline{\alpha}_J\|x-y\|^2$. By Lemma~\ref{l:L-resol}\ref{l:L-resol_mono}, we must have $\gamma\ell <1$, and then, by assumption, $\lambda\left(\nu+\frac{\lambda}{1-\gamma^2\ell^2}\right)\geq 0$. It now follows from \eqref{e:identityQ} and \eqref{e:L-resol_only} that
\begin{subequations}
\begin{align}
\|Qx-Qy\|^2
&\leq \nu^2\|x-y\|^2 +2\nu\lambda\scal{x-y}{Jx-Jy} \notag\\ 
&\qquad +\frac{\lambda^2}{1-\gamma^2\ell^2}(2\scal{x-y}{Jx-Jy} -\|x-y\|^2)\\
&=\Big(\nu^2-\frac{\lambda^2}{1-\gamma^2\ell^2}\Big)\|x-y\|^2 +2\lambda\Big(\nu+\frac{\lambda}{1-\gamma^2\ell^2}\Big)\scal{x-y}{Jx-Jy}\\
&\leq \Big(\nu^2-\frac{\lambda^2}{1-\gamma^2\ell^2}\Big)\|x-y\|^2 +2\lambda\Big(\nu+\frac{\lambda}{1-\gamma^2\ell^2}\Big)(1+\gamma\alpha)\overline{\alpha}_J\|x-y\|^2\\
&=\Big(\nu^2+\lambda\big(2\nu(1+\gamma\alpha)+\lambda\big)\overline{\alpha}_J\Big)\|x-y\|^2 =\overline{\rho}^2\|x-y\|^2.
\end{align}
\end{subequations}
Finally, we will prove $\overline{\alpha}_J\geq \alpha_J$, which implies $\overline{\rho}\leq \rho$, i.e., the Lipschitz constant is indeed improved. From the definition of $\alpha_J$, it suffices to consider the case in which $\alpha_J =1/((1+\gamma\alpha)(1+\gamma\ell))$. Then $\gamma\ell\leq 1$. 
Since $|\alpha|\leq \ell$ (see Remark~\ref{r:Lip_mono}), it holds that
\begin{subequations}
\begin{align}
0 <(1+\gamma\alpha)^2\leq 1+2\gamma\alpha+\gamma^2\ell^2
&\leq 
1+2\gamma\alpha+\gamma^2\ell^2
+(\gamma\ell-\gamma\alpha)-\gamma\ell(\gamma\ell-\gamma\alpha)\\
&=1+\gamma\alpha+\gamma\ell+(\gamma\alpha)(\gamma\ell)\\
&=(1+\gamma\alpha)(1+\gamma\ell),
\end{align}
\end{subequations}
and so 
\begin{equation}
\overline{\alpha}_J =\frac{1}{1+2\gamma\alpha+\gamma^2\ell^2}\geq \alpha_J =\frac{1}{(1+\gamma\alpha)(1+\gamma\ell)}.
\end{equation}
The proof is complete.
\end{proof}

\begin{remark}
In the setting of Proposition~\ref{p:cresol}\ref{p:cresol_Lip}, if $\nu\lambda\leq 0$, then one can also obtain a Lipschitz constant of $Q$ via Proposition~\ref{p:cresol}\ref{p:cresol_mono} and \eqref{e:expan} in Lemma~\ref{l:L-resol}\ref{l:L-resol_only}, in particular,
\begin{equation}
\|Qx-Qy\|^2\leq \left(\nu^2+\frac{\lambda\big(2\nu(1+\gamma\alpha)+\lambda\big)}{(1+\gamma\ell)^2}\right)\|x-y\|^2,
\end{equation}
i.e., $Q$ is Lipschitz continuous with constant
\begin{equation}
\rho':=\sqrt{\nu^2+\frac{\lambda\big(2\nu(1+\gamma\alpha)+\lambda\big)}{(1+\gamma\ell)^2}}.
\end{equation}
However, $\rho'$ is actually larger than $\rho$ in \eqref{e:LipQ}, which means that $\rho$ is a better Lipschitz constant than $\rho'$. To see this, since $\lambda(2\nu(1+\gamma\alpha)+\lambda)\leq 0$, we only need to check that $1/(1+\gamma\ell)^2 \leq \alpha_J$. Noting from Remark~\ref{r:Lip_mono} that $\alpha\leq \ell$, we have $0 <1+\gamma\alpha\leq 1+\gamma\ell$ and $0 <1+2\gamma\alpha+\gamma^2\ell^2\leq (1+\gamma\ell)^2$. Therefore,
\begin{equation}
\frac{1}{(1+\gamma\ell)^2}\leq \min\left\{\frac{1}{(1+\gamma\alpha)(1+\gamma\ell)}, \frac{1}{1+2\gamma\alpha+\gamma^2\ell^2}\right\}\leq \alpha_J.
\end{equation}
\end{remark}

\begin{corollary}[relaxed resolvents of $\alpha$-monotone operators]
\label{c:rresol}
Let $A\colon X\rightrightarrows X$ be $\alpha$-monotone and let $\gamma\in \RPP$. Suppose that $J:=J_{\gamma A}$ is single-valued and define $R:=(1-\lambda)\Id +\lambda J$ with $\lambda\in \RP$, and $Q:=R+\varepsilon\Id$ with $\varepsilon\in \RR$. Then the following hold:
\begin{enumerate}
\item\label{c:rresol_relax} 
If $\lambda\geq 1$, then, for all $x, y\in \dom J$, 
\begin{equation}
\|Rx-Ry\|^2\leq (\lambda-1)^2\|x-y\|^2 -\lambda\big((\lambda-1)(2+2\gamma\alpha)-\lambda\big)\|Jx-Jy\|^2.
\end{equation}
\item\label{c:rresol_perturb} 
If $\varepsilon\leq \lambda-1$, then, for all $x, y\in \dom J$, 
\begin{align}
\|Qx-Qy\|^2 &\leq (\lambda-1-\varepsilon)^2\|x-y\|^2 \notag\\ 
&\quad -\lambda\big((\lambda-1)(2+2\gamma\alpha)-\lambda-2\varepsilon(1+\gamma\alpha)\big)\|Jx-Jy\|^2.
\end{align}
Consequently, if additionally $(\lambda-1)(2+2\gamma\alpha)-\lambda-2\varepsilon(1+\gamma\alpha)\geq 0$, then $Q$ is Lipschitz continuous with constant $(\lambda-1-\varepsilon)$.
\end{enumerate} 
\end{corollary}
\begin{proof}
Because \ref{c:rresol_relax} is a consequence of \ref{c:rresol_perturb} with $\varepsilon =0$, it suffices to prove only the latter. To this end, noting that $Q =R+\varepsilon\Id =(1-\lambda+\varepsilon)\Id+\lambda J$ and using Proposition~\ref{p:cresol}\ref{p:cresol_mono} with $\nu =1-\lambda+\varepsilon \leq 0$, we have that
\begin{subequations}
\begin{align}
\|Qx-Qy\|^2&\leq (1-\lambda+\varepsilon)^2\|x-y\|^2 +\lambda\left(2(1-\lambda+\varepsilon)(1+\gamma\alpha)+\lambda\right)\|Jx-Jy\|^2\\
&=(\lambda-1-\varepsilon)^2\|x-y\|^2 -\lambda\big((\lambda-1)(2+2\gamma\alpha)-\lambda-2\varepsilon(1+\gamma\alpha)\big)\|Jx-Jy\|^2
\end{align}
\end{subequations}
which proves \ref{c:rresol_perturb}. 
\end{proof}

\begin{corollary}[relaxed resolvents of Lipschitz $\alpha$-monotone operators]
\label{c:L-rresol}
Let $A\colon X\to X$ be $\alpha$-monotone and Lipschitz continuous with constant $\ell$. Also let $\gamma\in\RPP$ and $\lambda\in\RPP$ be such that
\begin{equation}\label{e:cL-rresol-1}
1+\gamma\alpha>0
\ \text{ and }\ 
\lambda(1+2\gamma\alpha)-2(1+\gamma\alpha)\geq 0.
\end{equation}
Define $J:=J_{\gamma A}$, $R:=(1-\lambda)\Id +\lambda J$, and $Q :=\Id-\varepsilon R$ with $\varepsilon\in \RP$. Then the following hold:
\begin{enumerate}
\item\label{c:L-rresol_relax} 
$R$ is Lipschitz continuous with constant 
\begin{equation}
\sqrt{(\lambda-1)^2-\frac{\lambda\big((\lambda-1)(2+2\gamma\alpha)-\lambda\big)}{1+2\gamma\alpha+\gamma^2\ell^2}}.
\end{equation} 
\item\label{c:L-rresol_perturb} 
$Q$ is Lipschitz continuous with constant 
\begin{equation}
\sqrt{(1+\varepsilon(\lambda-1))^2 -\varepsilon\lambda\Big[2(1+\gamma\alpha) +\varepsilon\big(\lambda(1+2\gamma\alpha)-2(1+\gamma\alpha)\big)\Big]\alpha_J},
\end{equation}
where $\alpha_J$ is defined as \eqref{e:alpha_J}.
\end{enumerate}
\end{corollary}
\begin{proof}
\ref{c:L-rresol_relax}: We observe that $R=\nu\Id+\lambda J$ with $\nu :=1-\lambda$, that
\begin{equation}
\lambda(2\nu(1+\gamma\alpha)+\lambda) =-\lambda\Big(\lambda(1+2\gamma\alpha)-2(1+\gamma\alpha)\Big)\leq 0,
\end{equation} 
and that, whenever $\gamma\ell <1$, 
\begin{equation}
\nu+\frac{\lambda}{1-\gamma^2\ell^2}=
(1-\lambda)+\frac{\lambda}{1-\gamma^2\ell^2} >1-\lambda+\lambda =1 >0.
\end{equation}
Applying Proposition~\ref{p:cresol}\ref{p:cresol_Lip} to $R=\nu\Id+\lambda\Id$ implies that $R$ is Lipschitz continuous with constant
\begin{equation}
\sqrt{(1-\lambda)^2+\frac{\lambda\big(2(1-\lambda)(1+\gamma\alpha)+\lambda\big)}{1+2\gamma\alpha+\gamma^2\ell^2}},
\end{equation} 
which gives the claim.

\ref{c:L-rresol_perturb}: Using the first part of Proposition~\ref{p:cresol}\ref{p:cresol_Lip} and writing $Q =\widetilde{\nu}\Id+\widetilde{\lambda}J$ with $\widetilde{\nu} :=1+\varepsilon(\lambda-1)$ and $\widetilde{\lambda} :=-\varepsilon\lambda$, it suffices to check that
\begin{equation}\label{e:cL-rresol-2}
\widetilde{\lambda}(2\widetilde{\nu}(1+\gamma\alpha)+\widetilde{\lambda})=(-\varepsilon\lambda)\big( 2(1+\varepsilon(\lambda-1))(1+\gamma\alpha)+(-\varepsilon\lambda) \big)\leq 0.
\end{equation}
Indeed, we have that $\varepsilon\geq 0$, $\lambda >0$, and  
\begin{subequations}
\begin{align}
2(1+\varepsilon(\lambda-1))(1+\gamma\alpha) -\varepsilon\lambda 
&=2(1+\gamma\alpha) +\varepsilon\big(2(\lambda-1)(1+\gamma\alpha)-\lambda\big)\\
&=2(1+\gamma\alpha) +\varepsilon(\lambda(1+2\gamma\alpha)-2(1+\gamma\alpha)) >0
\end{align}
\end{subequations}
by \eqref{e:cL-rresol-1}. So \eqref{e:cL-rresol-2} holds and the conclusion follows.
\end{proof}

\section{Adaptive Douglas--Rachford algorithm}
\label{s:aDR}

Throughout this section, $A,B\colon X\rightrightarrows X$, $(\gamma,\delta,\lambda,\mu)\in \RPP^4$, and $\kappa\in \left]0,1\right[$. We define
\begin{subequations}\label{e:JRT}
\begin{align}
J_1&:=J_{\gamma A}=(\Id+\gamma A)^{-1}, 
&R_1&:=J_{\gamma A}^\lambda=(1-\lambda)\Id+\lambda J_1,\\
J_2&:=J_{\delta B}=(\Id+\delta B)^{-1},
&R_2&:=J_{\delta B}^\mu=(1-\mu)\Id+\mu J_2
\end{align}
and consider the \emph{adaptive DR operator} defined by
\begin{align}\label{e:aDR}
T:= (1-\kappa)\Id+ \kappa R_2R_1.
\end{align}
\end{subequations}  
For convenience of notation, we already drop the parameters $\lambda,\mu,\kappa$ and $A,B$ associated with the operators $J_1, R_1, J_2, R_2$, and $T$. When $(\lambda,\mu,\kappa)=(2,2,1/2)$, the operator $T$ in \eqref{e:aDR} reduces to the \emph{classical DR operator} \cite{DR56,LM79}. In fact, formulation \eqref{e:JRT} was previously used in \cite{DP18mor, DP18jogo} for feasibility problems, which allow for eliminating $\gamma$ and $\delta$ while choosing the parameters $\lambda$ and $\mu$ independently. However, such an advantage no longer exists for the case of general operators. In other words, all parameters $\gamma,\delta,\lambda$, and $\mu$ must satisfy a certain set of requirements simultaneously, as we will see shortly.

The adaptive DR operator is indeed motivated by the problem of finding a zero of the sum of two operators, that is,
\begin{equation}\label{e:sumprob}
\text{find $x\in X$ such that}\quad 0\in Ax+Bx.
\end{equation}
We also denote by
\begin{equation}
\zer(A+B) :=(A+B)^{-1}(0) =\menge{x\in X}{0\in Ax+Bx}
\end{equation}
the set of solutions of problem \eqref{e:sumprob}. Given a starting point $x_0\in X$, the \emph{adaptive DR algorithm} generates a sequence $(x_n)_\nnn$, also called a \emph{DR sequence}, by 
\begin{equation}\label{e:dr_seq}
\forall\nnn,\quad x_{n+1}\in Tx_n.
\end{equation}
Then, we expect the DR sequence $(x_n)_\nnn$ to converge to some point $\overline{x}\in\Fix T$ such that $J_1\overline{x}$ contains a solution to the original problem \eqref{e:sumprob}. For this purpose, we will require that
\begin{equation}
\label{e:CQ}
(\lambda-1)(\mu-1) =1 \quad\text{and}\quad 
\delta =(\lambda-1)\gamma, 
\end{equation}
which are also equivalent to $\lambda =\mu/(\mu-1)$ and $\gamma =(\mu-1)\delta$, respectively.
The next lemma shows the {\em necessity} of \eqref{e:CQ}.

\begin{lemma}[fixed points of adaptive DR operator]
\label{l:Fix}
The following statements hold: 
\begin{enumerate}
\item\label{l:Fix_Id-R2R1} 
$\Id-T =\kappa(\Id-R_2R_1)$.
\item\label{l:Fix_J1-J2R1} 
Suppose that $(\lambda-1)(\mu-1) =1$. Then
\begin{equation}
\label{e:Tset-valued}
\forall x\in \dom T, \quad (\Id-T)x =\menge{\kappa\mu\left(a-J_2\left((1-\lambda)x+\lambda a\right)\right)}{a\in J_1x}.
\end{equation}
Consequently, if $J_1$ is single-valued, then $\Id-T =\kappa\mu(J_1-J_2R_1)$.
\item\label{l:Fix_sol} 
Suppose that \eqref{e:CQ} holds. Then $\Fix T\neq \varnothing$ if and only if $\zer(A+B)\neq \varnothing$. Moreover, if $J_1$ is single-valued, then
\begin{equation}
J_1(\Fix T) =\zer(A+B).
\end{equation}
\end{enumerate}
\end{lemma}
\begin{proof}
\ref{l:Fix_Id-R2R1}: This is clear from the definition of $T$.

\ref{l:Fix_J1-J2R1}: Let $x\in \dom T$. Noting that $(\lambda-1)(\mu-1) =1$ also implies $\lambda(\mu-1) =\mu$, we have  
\begin{subequations}
\begin{align}
(\Id-R_2R_1)x &=\menge{x -R_2\big((1-\lambda)x +\lambda a\big)}{a\in J_1x} \\
&=\menge{x- (1-\mu)\big((1-\lambda)x +\lambda a\big) -\mu J_2\big((1-\lambda)x +\lambda a\big)}{a\in J_1x} \\
&=\menge{\mu\left(a-J_2\left((1-\lambda)x+\lambda a\right)\right)}{a\in J_1x}.
\end{align}
\end{subequations}
This together with \ref{l:Fix_Id-R2R1} proves \eqref{e:Tset-valued}, from which the remaining conclusion follows.

\ref{l:Fix_sol}: We derive from the assumption and \ref{l:Fix_J1-J2R1} that
\begin{subequations}
\begin{align}
x\in \Fix T 
&\iff 0\in (\Id-T)x \\
&\iff \exists a\in J_1x,\quad a\in J_2\big((1-\lambda)x+\lambda a\big) \\
&\iff \exists a\in J_1x,\quad (1-\lambda)(x-a)\in \delta Ba \\
&\iff \exists a\in X,\quad x-a\in \gamma Aa \text{~and~} -(x-a)\in \gamma Ba \\
&\iff \exists a\in J_1x,\quad 0\in \gamma Aa+\gamma Ba \\
&\iff \exists a\in J_1x\cap \zer(A+B),
\end{align}
\end{subequations}
which completes the proof.
\end{proof}

As shown in Lemma~\ref{l:Fix}, a solution of \eqref{e:sumprob} can be found by means of fixed points of the adaptive DR operator. Therefore, our analysis will mainly revolve around the convergence to the fixed points under the condition \eqref{e:CQ}.

\subsection{Convergence via Fej\'er monotonicity}

Recall that a sequence $(x_n)_\nnn$ is said to be \emph{Fej\'er monotone} with respect to a nonempty subset of $C$ of $X$ if 
\begin{equation}
\forall c\in C,\ \forall\nnn, \quad \|x_{n+1}-c\|\leq \|x_n-c\|.
\end{equation}
The use of Fej\'er monotonicity is quite common in the convergence theory of monotone operators.
In the following abstract convergence result, our analysis relies on the Fej\'er monotonicity of DR sequences generated by the adaptive DR operator $T$ with respect to $\Fix T$ and does not require the nonexpansiveness of $R_2R_1$.

\begin{theorem}[abstract convergence]
\label{t:gencvg}
Let $(\omega_1, \omega_2, \omega_3)\in \RR^3$ such that
\begin{subequations}
\label{e:omegas}
\begin{alignat}{2}
&\text{either}\quad &&\big\{\omega_2 =\omega_3 =0 \text{~and~} \omega_1 >0\big\};\\
&\text{or}\quad &&\big\{\omega_2 +\omega_3> 0 \text{~and~} \omega_1 +\frac{\omega_2\omega_3}{\kappa^2\mu^2(\omega_2+\omega_3)} >0\big\}.
\end{alignat}
\end{subequations}
Suppose that $(\lambda-1)(\mu-1) =1$, that $J_1$ and $J_2$ are single-valued, that $\Fix T\neq \varnothing$, and that, for all $x\in \dom T$, $y\in \Fix T$,
\begin{align}
\label{e:qFejer}
\|Tx-y\|^2&\leq \|x-y\|^2 -\omega_1\|(\Id-T)x\|^2 \notag \\
&\qquad -\omega_2\|J_1x-J_1y\|^2 -\omega_3\|J_2R_1x-J_2R_1y\|^2. 
\end{align}
Let $(x_n)_\nnn\subset\dom T$ be a DR sequence generated by $T$. 
Then $(x_n)_\nnn$ converges weakly to a point $\overline{x}\in \Fix T$. Furthermore, the following hold:
\begin{enumerate}
\item\label{t:gencvg_strong} 
If $\omega_2+ \omega_3> 0$, then the shadow sequences $(J_1x_n)_\nnn$  and $(J_2R_1x_n)_\nnn$ converge strongly to $J_1\overline{x}$ and $J_1(\Fix T) =J_2R_1(\Fix T) =\{J_1\overline{x}\}$.
\item\label{t:gencvg_asymp} 
If $T$ is nonexpansive, then the rate of asymptotic regularity of $T$ is $o(1/\sqrt{n})$, i.e., $\|(\Id-T)x_n\| =o(1/\sqrt{n})$ as $n\to +\infty$.
\item\label{t:gencvg_Lip}
If, for all $x,y\in \dom T$, 
\begin{align}
\label{e:LipT}
\|Tx-Ty\|^2&\leq \|x-y\|^2 -\omega_1\|(\Id-T)x-(\Id-T)y\|^2 \notag \\
&\qquad -\omega_2\|J_1x-J_1y\|^2 -\omega_3\|J_2R_1x-J_2R_1y\|^2, 
\end{align}
then \eqref{e:qFejer} holds for all $x\in \dom T$, $y\in \Fix T$, and $T$ is nonexpansive.
\end{enumerate}
\end{theorem}
\begin{proof}
Define
\begin{equation}
\omega'_2:= \begin{cases}
\frac{\omega_2\omega_3}{\omega_2+\omega_3} &\text{if~} \omega_2+\omega_3> 0, \\
0 &\text{if~} \omega_2 =\omega_3 =0
\end{cases}
\quad\text{and}\quad
\omega'_3:= \begin{cases}
\frac{1}{\omega_2+\omega_3} &\text{if~} \omega_2+\omega_3> 0, \\
0 &\text{if~} \omega_2 =\omega_3 =0. 
\end{cases}
\end{equation}
Then 
\begin{equation}
\omega_1+\frac{\omega'_2}{\kappa^2\mu^2} >0 \quad\text{and}\quad \omega'_3\geq 0.
\end{equation}
For all $x,y\in \dom T$, we derive from \eqref{e:identity'} and Lemma~\ref{l:Fix}\ref{l:Fix_J1-J2R1} that
\begin{subequations}
\label{e:J1_J2R1}
\begin{align}
&\omega_2\|J_1x -J_1y\|^2 +\omega_3\|J_2R_1x -J_2R_1y\|^2 
\\={}&\omega'_2\|(J_1-J_2 R_1)x -(J_1-J_2 R_1)y\|^2 +\omega'_3\|\omega_2\big( J_1x -J_1y \big) +\omega_3\big( J_2R_1x -J_2R_1y \big)\|^2 \\
={}&\frac{\omega'_2}{\kappa^2\mu^2}\|(\Id-T)x-(\Id-T)y\|^2 +\omega'_3\|\omega_2\big( J_1x -J_1y \big) +\omega_3\big( J_2R_1x -J_2R_1y \big)\|^2.
\end{align}
\end{subequations}
Combining with the assumption on $T$ implies that, for all $x\in \dom T$, $y\in \Fix T$,
\begin{align}
\|Tx-y\|^2&\leq \|x-y\|^2 -\big(\omega_1+\frac{\omega'_2}{\kappa^2\mu^2}\big)\|(\Id-T)x\|^2 \notag \\
&\qquad -\omega'_3\|\omega_2\big( J_1x -J_1y \big) +\omega_3\big( J_2R_1x -J_2R_1y \big)\|^2.
\end{align}
Therefore, for all $\nnn$ and all $y\in \Fix T$,
\begin{align}
\label{e:qFejer'}
\|x_{n+1}-y\|^2&\leq \|x_n-y\|^2 -\big(\omega_1+\frac{\omega'_2}{\kappa^2\mu^2}\big)\|(\Id-T)x_n\|^2 \notag\\ &\quad-\omega'_3\|\omega_2\big( J_1x_n -J_1y \big) +\omega_3\big( J_2R_1x_n -J_2R_1y \big)\|^2. 
\end{align}
We deduce that $(x_n)_\nnn$ is Fej\'er monotone with respect to $\Fix T$ and hence bounded. By the telescoping technique, for all $y\in \Fix T$,
\begin{align}
\label{e:telescope}
&\big(\omega_1+\frac{\omega'_2}{\kappa^2\mu^2}\big)\sum_{n=0}^{+\infty} \|(\Id-T)x_n\|^2  
+\omega'_3\sum_{n=0}^{+\infty} \|\omega_2\big( J_1x_n -J_1y \big) +\omega_3\big( J_2R_1x_n -J_2R_1y \big)\|^2 \notag\\ 
&\leq \|x_0-y\|^2< +\infty.
\end{align} 
Since $\omega_1+\frac{\omega'_2}{\kappa^2\mu^2}> 0$, it follows that
\begin{equation}
\label{e:asymptotic}
(\Id-T)x_n\to 0\quad\text{as~} n\to +\infty.
\end{equation}
Now let $x^*$ be a weak cluster point of $(x_n)_\nnn$. Then there exists a subsequence $(x_{k_n})_\nnn$ of $(x_n)_\nnn$ such that $x_{k_n}\rightharpoonup x^*$. By \eqref{e:asymptotic}, $(\Id-T)x_{k_n}\to 0$, and by \cite[Corollary~4.28]{BC17}, $x^*\in \Fix T$. In turn, \cite[Theorem~5.5]{BC17} implies that $(x_n)_\nnn$ converges weakly to a point $\overline{x}\in \Fix T$.

\ref{t:gencvg_strong}: If $\omega_2+ \omega_3> 0$, then $\omega'_3= 1/(\omega_2+ \omega_3)> 0$ and, by \eqref{e:telescope}, for all $y\in \Fix T$, 
\begin{equation}
\omega_2\big( J_1x_n -J_1y \big) +\omega_3\big( J_2R_1x_n -J_2R_1y \big)\to 0.
\end{equation}
Together with
\begin{equation}
\big( J_1x_n -J_1y \big) -\big( J_2R_1x_n -J_2R_1y \big)
= \frac{1}{\kappa\mu}\big( (\Id-T)x_n- (\Id-T)y \big)= \frac{1}{\kappa\mu}(\Id-T)x_n\to 0,
\end{equation}
we obtain  
\begin{equation}
J_1x_n\to J_1y \quad\text{and}\quad J_2R_1x_n\to J_2R_1y =J_1y,
\end{equation}
which also means that $J_1(\Fix T) =J_2R_1(\Fix T) =\{J_1\overline{x}\}$.

\ref{t:gencvg_asymp}: By the nonexpansiveness of $T$,
\begin{equation}
\forall\nnn,\quad \|(\Id-T)x_{n+1}\| =\|Tx_n-Tx_{n+1}\|\leq \|x_n-x_{n+1}\| =\|(\Id-T)x_n\|.
\end{equation}
Combining with \eqref{e:telescope}, we obtain that
\begin{equation}
\frac{n}{2}\|(\Id-T)x_n\|^2\leq \sum_{k=\lfloor n/2 \rfloor}^n \|(\Id-T)x_k\|^2\to 0 \quad\text{as~} n\to +\infty,
\end{equation}
where $\lfloor n/2 \rfloor$ is the largest integer not exceeding $n/2$. The conclusion then follows.

\ref{t:gencvg_Lip}: Assume that \eqref{e:LipT} holds for all $x,y\in \dom T$. Then \eqref{e:qFejer} holds for all $x\in \dom T$, $y\in \Fix T$ since $(\Id-T)y =0$ in this case. Next, it follows from \eqref{e:LipT} and \eqref{e:J1_J2R1} that, for all $x,y\in \dom T$, 
\begin{subequations}
\begin{align}
\|Tx-Ty\|^2&\leq \|x-y\|^2 -\big(\omega_1+\frac{\omega'_2}{\kappa^2\mu^2}\big)\|(\Id-T)x-(\Id-T)y\|^2 \notag \\
&\qquad -\omega'_3\|\omega_2\big( J_1x -J_1y \big) +\omega_3\big( J_2R_1x -J_2R_1y \big)\|^2 \\
&\leq \|x-y\|^2,
\end{align}
\end{subequations}
which completes the proof.
\end{proof}

The following result provides a quantitative measurement for the adaptive DR operator, which is important for our analysis.

\begin{proposition}[metric inequality for adaptive DR operator]
\label{p:estimate}
Suppose that $A$ and $B$ are respectively $\alpha$- and $\beta$-monotone, that \eqref{e:CQ} holds and $\min\{\lambda, \mu\}\geq 1$, and that $J_1$ and $J_2$ are single-valued. Then 
for all $x, y\in \dom T$,
\begin{align}
\|T x-T y\|^2
&\leq \|x-y\|^2 -\frac{1-\kappa}{\kappa}\|(\Id-T)x-(\Id-T)y\|^2 \notag\\
&\quad -\kappa\mu(2+2\gamma\alpha-\mu)\|J_1x-J_1y\|^2  \notag\\ 
&\quad -\kappa\mu\big(\mu-(2-2\gamma\beta)\big)\|J_2R_1x-J_2R_1y\|^2.
\end{align}
\end{proposition}
\begin{proof}
Let $x,y\in \dom R_2R_1 =\dom T$. We observe from \eqref{e:identity} and Lemma~\ref{l:Fix}\ref{l:Fix_Id-R2R1} that
\begin{subequations}
\label{e:identityT}
\begin{align}
\|Tx-Ty\|^2
&= (1-\kappa)\|x-y\|^2+ \kappa\|R_2R_1 x - R_2R_1 y\|^2\\ 
&\hspace{1.24in} -\kappa(1-\kappa)\|(\Id-R_2R_1)x- (\Id-R_2R_1)y\|^2\notag\\
&= (1-\kappa)\|x-y\|^2
-\frac{1-\kappa}{\kappa}\|(\Id-T)x- (\Id-T)y\|^2
+\kappa\|R_2R_1 x - R_2R_1 y\|^2.
\end{align}
\end{subequations}
Next, applying Corollary~\ref{c:rresol}\ref{c:rresol_relax} first to $R_2$ and then to $R_1$ yields
\begin{subequations}
\begin{align}
\|R_2R_1x-R_2R_1y\|^2&
\leq
(\mu-1)^2(\lambda-1)^2\|x-y\|^2\\
&\quad-(\mu-1)^2
\lambda\Big((\lambda-1)(2+2\gamma\alpha)-\lambda\Big)
\|J_1x-J_1y\|^2\\
&\quad
-\mu\Big((\mu-1)(2+2\delta\beta)-\mu\Big)\|J_2R_1x-J_2R_1y\|^2\\
&=:\eta_0\|x-y\|^2 -\eta_1\|J_1x-J_1y\|^2 -\eta_2\|J_2R_1x-J_2R_1y\|^2.
\end{align}
\end{subequations}
Now, it follows from \eqref{e:CQ} that
\begin{subequations}
\begin{align}
\eta_0
&=\big((\mu-1)(\lambda-1)\big)^2 =1,\\
\eta_1
&=(\mu-1)^2(\lambda-1)^2\frac{\lambda}{\lambda-1}\Big((2+2\gamma\alpha)-\frac{\lambda}{\lambda-1}\Big) 
=\mu\big(2+2\gamma\alpha-\mu\big),\\
\eta_2
&=\mu\Big(2(\mu-1)+2\gamma\beta-\mu\Big)
=\mu\big(\mu-(2-2\gamma\beta)\big).
\end{align}
\end{subequations}
Altogether, we get the conclusion.
\end{proof}

So far in this section, we have often assumed single-valuedness of the resolvents $J_1$ and $J_2$, which leads to the same property for the adaptive DR operator $T$. Indeed, since either $A$ or $B$ may not necessarily be monotone, the single-valuedness is not guaranteed. Nevertheless, the choice of parameters can help clearing up the issue as seen in the following lemma, which is based on Proposition~\ref{p:resol}. We will further establish that, given suitable $\alpha$- and $\beta$-monotone operators, it is always possible to choose parameters $(\gamma,\delta,\lambda,\mu)\in \RPP^2\times \left]1,+\infty\right[^2$ so that \emph{all} objectives are met: the adaptive DR operator enjoys the single-valuedness and full domain properties; \eqref{e:CQ} is satisfied; and every DR sequence converges to a fixed point via which problem \eqref{e:sumprob} is solved.

\begin{lemma}[single-valuedness and full domain of adaptive DR operator]
\label{l:domT}
Suppose that $A$ and $B$ are maximally $\alpha$- and $\beta$-monotone with $\alpha+\beta\geq 0$. 
Then there exists $(\gamma,\delta,\lambda,\mu)\in \RPP^2\times \left]1,+\infty\right[^2$ such that
\begin{subequations}
\label{e:parameters}
\begin{align}
&1+2\gamma\alpha >0,
\label{e:parameters_a}\\
&\mu\in \left[2-2\gamma\beta, 2+2\gamma\alpha\right],
\label{e:parameters_b} \\
&(\lambda-1)(\mu-1) =1, \text{~and~} \delta =(\lambda-1)\gamma. \label{e:parameters_c}
\end{align}
\end{subequations}
Moreover, \eqref{e:parameters} implies that $\min\{1+\gamma\alpha, 1+\delta\beta\} >0$ and that $J_1$, $J_2$, and $T$ are single-valued and have full domain. 
\end{lemma}
\begin{proof}
To show the existence, we first take $\gamma>0$ such that $1/\gamma >-2\alpha$. Then $1+2\gamma\alpha >0$ and $2+2\gamma\alpha =1+(1+2\gamma\alpha) >1$.
Using $\alpha+\beta\geq 0$, we derive that
\begin{equation}
2+2\gamma\alpha =2\gamma(\alpha+\beta) +(2-2\gamma\beta)\geq 2-2\gamma\beta.
\end{equation}
Hence, we can always choose $\mu >1$ satisfying \eqref{e:parameters_b}. Next, with such $\mu$, we define $\lambda :=\mu/(\mu-1) =1+1/(\mu-1) >1$ and $\delta :=(\lambda-1)\gamma$. Then \eqref{e:parameters_c} is clearly satisfied. 

Now, take any $(\gamma,\delta,\lambda,\mu)$ satisfying \eqref{e:parameters}. We have
\begin{equation}
1+\delta\beta =(\lambda-1)(\mu-1)+(\lambda-1)\gamma\beta =\frac{1}{2}(\lambda-1)(\mu + \mu - (2-2\gamma\beta)) >0.
\end{equation}
Thus, $\min\{1+\gamma\alpha, 1+\delta\beta\} >0$. The remaining conclusion follows from Proposition~\ref{p:resol}. 
\end{proof}

We are now ready to state our convergence results for the adaptive DR algorithm.
\begin{theorem}[adaptive DR algorithm for $\alpha$- and $\beta$-monotone operators]
\label{t:DRcvg}
Suppose that $A$ and $B$ are respectively maximally $\alpha$- and $\beta$-monotone with $\zer(A+B)\neq \varnothing$, and that one of the following holds:
\begin{enumerate}
\item\label{t:DRcvg_adapt}
(Adaptive DR algorithm)
$\alpha+\beta\geq 0$ and $(\gamma,\delta,\lambda,\mu)\in \RPP^2\times \left]1,+\infty\right[^2$ satisfies \eqref{e:parameters}.

\item\label{t:DRcvg_DR} 
(Classical DR algorithm) $\lambda =\mu =2$, $\gamma=\delta\in \RPP$, and 
\begin{subequations}
\label{e:DRcvg_DR}
\begin{alignat}{2}
&\text{either}\quad &&\alpha =\beta =0;\\
&\text{or}\quad &&\alpha+\beta >0 \text{~and~} 1+\gamma\frac{\alpha\beta}{\alpha+\beta} >\kappa. \label{e:DRcvg_DR_b}
\end{alignat}
\end{subequations}
\end{enumerate} 
Then every DR sequence $(x_n)_\nnn$ generated by $T$ converges weakly to a point $\overline{x}\in \Fix T$ with $J_1\overline{x}\in \zer(A+B)$ and the rate of asymptotic regularity of $T$ is $o(1/\sqrt{n})$. Moreover, if $\alpha+\beta >0$, then the shadow sequences $(J_1x_n)_\nnn$ and $(J_2R_1x_n)_\nnn$ converge strongly to $J_1\overline{x}$ and $\zer(A+B) =\{J_1\overline{x}\}$. 
\end{theorem}
\begin{proof}
We first observe that if \ref{t:DRcvg_adapt} holds, then, by Lemma~\ref{l:domT}, 
\begin{equation}
\label{e:domT}
\min\{1+\gamma\alpha,1+\delta\beta\} >0.
\end{equation}
Let us show that \eqref{e:domT} is also satisfied when \ref{t:DRcvg_DR} holds. Indeed, if $\alpha =\beta =0$, then \eqref{e:domT} is obvious. Otherwise, it follows from $\alpha+\beta >0$ and $1+\gamma\frac{\alpha\beta}{\alpha+\beta} >\kappa >0$ that
\begin{equation}
1+\gamma\alpha =\left(1+\gamma\frac{\alpha\beta}{\alpha+\beta}\right) +\gamma\frac{\alpha^2}{\alpha+\beta} >0
\end{equation}
and that 
\begin{equation}
1+\delta\beta =1+\gamma\beta =\left(1+\gamma\frac{\alpha\beta}{\alpha+\beta}\right) +\gamma\frac{\beta^2}{\alpha+\beta} >0.
\end{equation}
Thus, \eqref{e:domT} holds for all cases.

From \eqref{e:domT} and Proposition~\ref{p:resol}, we have that $J_1$ and $J_2$ are single-valued and have full domain, so does $T$.
Now by Proposition~\ref{p:estimate}, for all $x,y\in X$,
\begin{align}
\|T x-T y\|^2
&\leq \|x-y\|^2 -\omega_1\|(\Id-T)x-(\Id-T)y\|^2 \notag\\
&\quad -\omega_2\|J_1x-J_1y\|^2 -\omega_3\|J_2R_1x-J_2R_1y\|^2
\end{align}
with $\omega_1 :=(1-\kappa)/\kappa >0$, $\omega_2 :=\kappa\mu(2+2\gamma\alpha-\mu)$, $\omega_3 :=\kappa\mu\big(\mu-(2-2\gamma\beta)\big)$. 
Next, since $\zer(A+B)\neq \varnothing$, Lemma~\ref{l:Fix}\ref{l:Fix_sol} yields $\Fix T\neq \varnothing$. In view of Theorem~\ref{t:gencvg}, it suffices to verify assumption \eqref{e:omegas}. If \ref{t:DRcvg_adapt} holds, then, by \eqref{e:parameters}, $\omega_2, \omega_3\geq 0$, so \eqref{e:omegas} is satisfied; if \ref{t:DRcvg_DR} holds, then $\omega_2 =4\kappa\gamma\alpha$, $\omega_3 =4\kappa\gamma\beta$, and \eqref{e:omegas} holds due to \eqref{e:DRcvg_DR}. The proof is complete.
\end{proof}

\begin{remark}[under- and over-reflecting the resolvents]
Let us consider problem \eqref{e:sumprob} with $A$ and $B$ respectively maximally $\alpha$ and $(-\alpha)$-monotone for some $\alpha>0$. Recall that the classical DR algorithm uses the {\em exact} reflections of the resolvents (i.e., $\lambda=\mu=2$) if both operators are monotone. This is not applicable in this situation since $A$ is strongly monotone while $B$ is weakly monotone. Therefore, in order to guarantee the convergence,
the adaptive DR algorithm requires the choice $\mu =2+2\gamma\alpha>2$ (Theorem~\ref{t:DRcvg}\ref{t:DRcvg_adapt}), and thus $\lambda =\mu/(\mu-1) =1+1/(1+2\gamma\alpha)<2$. That means, we must under-reflect ($\lambda<2$) the resolvent of $A$, the strongly monotone operator, and over-reflect ($\mu>2$) the resolvent of $B$, the weakly monotone one. This phenomenon is somewhat counterintuitive, since in order to preserve nonexpansiveness, one would naturally think of doing the opposite, i.e., over-reflecting the resolvent of the strongly monotone operator and under-reflecting that of the weakly one.

While Theorem~\ref{t:DRcvg}\ref{t:DRcvg_adapt} is \emph{new}, Theorem~\ref{t:DRcvg}\ref{t:DRcvg_DR} not only unifies and simplifies but also extends Theorems~4.4 and 4.6 in \cite{GHY17} to the context of operators in Hilbert spaces (here we note that the condition (3.4) in \cite{GHY17} implies the second condition in \eqref{e:DRcvg_DR_b}). Moreover, the proof for the rate of asymptotic regularity of $T$ in Theorem~\ref{t:DRcvg}, which follows from Theorem~\ref{t:gencvg}\ref{t:gencvg_asymp}--\ref{t:gencvg_Lip}, is {\em simpler} than the treatment presented in \cite[Theorems~5.1 and 5.2]{GHY17}.
\end{remark}

The following result is an immediate corollary of Theorem~\ref{t:DRcvg}, in which we note that the adaptive DR algorithm reduces to the classical one when choosing $\mu=2$.

\begin{corollary}[one monotone and one strongly monotone operators]
\label{c:mono_strgmono}
Let $\alpha\in \RP$ and $\gamma\in\RPP$. Suppose that $A$ and $B$ are maximally monotone and that either
\begin{enumerate}
\item $A$ is $\alpha$-monotone and $\mu\in \left[2,2+2\gamma\alpha\right]$, or
\item $B$ is $\alpha$-monotone and $\mu\in \left[2-2\gamma\alpha,2\right]$.
\end{enumerate}
Suppose also that $\zer(A+B)\neq \varnothing$ and that $\lambda =\mu/(\mu-1)$ and $\delta =(\mu-1)\gamma$.
Then every DR sequence $(x_n)_\nnn$ generated by $T$ converges weakly to a point $\overline{x}\in \Fix T$ with $J_1\overline{x}\in \zer(A+B)$  and the rate of asymptotic regularity of $T$ is $o(1/\sqrt{n})$. Moreover, if $\alpha >0$, then the shadow sequences $(J_1x_n)_\nnn$ and $(J_2R_1x_n)_\nnn$ converge strongly to $J_1\overline{x}$ and $\zer(A+B) =\{J_1\overline{x}\}$.
\end{corollary}
\begin{proof}
By Proposition~\ref{p:maxalpha}\ref{p:maxalpha_mono}, we readily have that either $A$ or $B$ is maximally $\alpha$-monotone while the other is maximally monotone.
Now, apply Theorem~\ref{t:DRcvg}\ref{t:DRcvg_adapt} with $(\alpha,\beta)$ replaced by $(\alpha,0)$ or $(0,\alpha)$.
\end{proof}

\subsection{Linear convergence under Lipschitz assumption}

In this section, we provide linear convergence results for the adaptive DR algorithm for $\alpha$- and $\beta$-monotone operators when, in addition, one operator is Lipschitz continuous. Comparing with \cite{Gis17,MV18}, our work indeed gives a new perspective on this topic by using {\em adaptive} parameters. Moreover, we improve the linear convergence rate obtained by \cite{MV18} for the classical DR algorithm for a Lipschitz monotone and a strongly monotone operator (see Remark~\ref{r:L_improved}).

Recall that a sequence $(x_n)_\nnn$ converges to $\overline{x}$ with \emph{$Q$-linear} (or simply \emph{linear}) rate $\rho\in\left[0,1\right[$ if
\begin{equation}
\forall\nnn,\quad \|x_{n+1}-\overline{x}\|\leq \rho\|x_n-\overline{x}\|.
\end{equation}

\begin{theorem}[linear convergence when $A$ is Lipschitz]
\label{t:linear}
Suppose that either
\begin{enumerate}
\item\label{t:linear_full} 
$A$ is $\alpha$-monotone and Lipschitz continuous with constant $\ell$, $B$ is maximally $\beta$-monotone, and $\alpha+\beta >0$; or
\item\label{t:linear_reduce} 
$A$ is Lipschitz continuous with constant $\ell$, $B$ is maximally $\beta$-monotone with $\beta >\ell$, and $\alpha :=-\ell$.
\end{enumerate}
Suppose also that $\zer(A+B)\neq \varnothing$ and that $(\gamma,\delta,\lambda,\mu)\in\RPP^2\times \left]1,+\infty\right[^2$ satisfies \eqref{e:parameters}.
Then $T$ is Lipschitz continuous with constant
\begin{equation}
\rho :=(1-\kappa)
\sqrt{\big(1+\varepsilon_1(\lambda-1)\big)^2-\varphi\alpha_J} +\kappa(1-\varepsilon(\lambda-1))\sqrt{1-\frac{\mu(2+2\gamma\alpha-\mu)}{1+2\gamma\alpha+\gamma^2\ell^2}} <1,
\end{equation} 
where
\begin{subequations}
\begin{align}
&\varepsilon :=\frac{\mu-(2-2\gamma\beta)}{2(1+\delta\beta)},\ \varepsilon_1:=\frac{\kappa\varepsilon}{1-\kappa},\\
&\varphi :=\varepsilon_1\lambda[2(1+\gamma\alpha)+\varepsilon_1\big(\lambda(1+2\gamma\alpha)-2(1+\gamma\alpha)\big)],\\
&\text{$\alpha_J$ as in \eqref{e:alpha_J}}.
\end{align}
\end{subequations}
Consequently, every DR sequence $(x_n)_\nnn$ generated by $T$ converges strongly to the unique fixed point $\overline{x}$ of $T$ with linear rate $\rho$.
\end{theorem}
\begin{proof}
In view of Remark~\ref{r:Lip_mono}, assumption \ref{t:linear_reduce} implies assumption \ref{t:linear_full} because if $A$ is Lipschitz continuous with constant $\ell$, then $A$ is also $\alpha$-monotone with $\alpha :=-\ell$. It thus suffices to assume \ref{t:linear_full}.  
First, Proposition~\ref{p:maxalpha}\ref{p:maxalpha_cont} implies that $A$ is maximally $\alpha$-monotone. Next, we learn from Lemma~\ref{l:domT} that 
\begin{equation}
\label{e:domT'}
\min\{1+\gamma\alpha,1+\delta\beta\} >0
\end{equation}
and that all operators $J_1,J_2$, and $T$ are single-valued and have full domain. 

By the choice of $\mu$, it holds that $0 <\mu-1 \leq 1+2\gamma\alpha$, and so
\begin{equation}
\lambda =1+\frac{1}{\mu-1}\geq 1+\frac{1}{1+2\gamma\alpha} =\frac{2(1+\gamma\alpha)}{1+2\gamma\alpha},
\end{equation} 
which yields
\begin{equation}
\lambda(1+2\gamma\alpha)-2(1+\gamma\alpha)\geq 0.
\end{equation}
From $\mu\geq 2-2\gamma\beta$, we have that $\varepsilon\geq 0$ and $\varepsilon_1\geq 0$. It follows that $\varphi\geq 0$ and that
\begin{equation}\label{e:180823a}
\varphi =0 \iff \varepsilon_1 =0 \iff \varepsilon =0 \iff \mu =2-2\gamma\beta.
\end{equation}
Define $Q_1 :=\Id -\varepsilon_1 R_1$. Using Corollary~\ref{c:L-rresol} and noting that $\lambda =\mu(\lambda-1)$, we derive that $R_1$ is Lipschitz continuous with constant  
\begin{equation}
\sqrt{(\lambda-1)^2-\frac{\lambda\big((\lambda-1)(2+2\gamma\alpha)-\lambda\big)}{1+2\gamma\alpha+\gamma^2\ell^2}} =(\lambda-1)\sqrt{1-\frac{\mu(2+2\gamma\alpha-\mu)}{1+2\gamma\alpha+\gamma^2\ell^2}},
\end{equation} 
and that $Q_1$ is Lipschitz continuous with constant 
\begin{equation}
\label{e:rho1}
\rho_1 :=\sqrt{(1+\varepsilon_1(\lambda-1))^2 -\varphi\alpha_J}\leq 1+\varepsilon_1(\lambda-1),
\end{equation}
where $\alpha_J$ is defined as in \eqref{e:alpha_J}. It follows from \eqref{e:180823a} that the inequality is strict whenever $\mu >2-2\gamma\beta$.

Next, define $Q_2 :=R_2+\varepsilon\Id$. 
Since $\gamma =(\mu-1)\delta$, we note that
\begin{equation}
\varepsilon =\frac{\mu-(2-2\gamma\beta)}{2(1+\delta\beta)} =\frac{(\mu-1)(2+2\delta\beta)-\mu}{2(1+\delta\beta)} <\mu-1,
\end{equation}
which also gives
\begin{equation}
(\mu-1)(2+2\delta\beta)-\mu-2\varepsilon(1+\delta\beta) =0.
\end{equation}
By Corollary~\ref{c:rresol}\ref{c:rresol_perturb}, $Q_2$ is Lipschitz continuous with constant $(\mu-1-\varepsilon)$. Combining with the Lipschitz continuity of $R_1$ and noting that $(\mu-1)(\lambda-1) =1$, we have that $Q_2R_1$ is Lipschitz continuous with constant
\begin{subequations}
\label{e:rho2}
\begin{align}
\rho_2
&:=(\mu-1-\varepsilon)(\lambda-1)\sqrt{1-\frac{\mu(2+2\gamma\alpha-\mu)}{1+2\gamma\alpha+\gamma^2\ell^2}}\\
&\phantom{:}=(1-\varepsilon(\lambda-1))\sqrt{1-\frac{\mu(2+2\gamma\alpha-\mu)}{1+2\gamma\alpha+\gamma^2\ell^2}}\\
&\phantom{:}\leq 1-\varepsilon(\lambda-1),
\end{align}
\end{subequations} 
where the inequality is strict whenever $\mu <2+2\gamma\alpha$.  

Now, we express
\begin{subequations}
\begin{align}
T&=(1-\kappa)\Id -(1-\kappa)\varepsilon_1R_1 + \kappa R_2R_1+\kappa\varepsilon R_1\\
&=(1-\kappa)(\Id-\varepsilon_1R_1)+\kappa(R_2+\varepsilon\Id)R_1\\
&=(1-\kappa)Q_1+\kappa Q_2R_1.
\end{align}
\end{subequations}
We note from $\alpha+\beta >0$ that $2+2\gamma\alpha >2-2\gamma\beta$, so at least one of two inequalities in \eqref{e:rho1} and \eqref{e:rho2} is strict.
Therefore, $T$ is Lipschitz continuous with constant
\begin{equation}
\rho :=(1-\kappa)\rho_1+\kappa \rho_2 <(1-\kappa)(1+\varepsilon_1(\lambda-1))+\kappa(1-\varepsilon(\lambda-1))=1,
\end{equation}
which completes the proof.
\end{proof}

The following is a direct consequence of Theorem~\ref{t:linear}, which was also proved in \cite{Gis17}.
\begin{corollary}[{\rm \cite[Theorem~6.5]{Gis17}}]
\label{c:classicalDR_gis}
Suppose that $A$ is $\alpha$-monotone with $\alpha\in \RPP$ and Lipschitz continuous with constant $\ell$, that $B$ is maximally monotone, and that $\zer(A+B)\neq \varnothing$. Suppose also that $\lambda =\mu =2$ and $\gamma =\delta\in \RPP$. 
Then $T$ is Lipschitz continuous with constant
\begin{equation}
\rho:=(1-\kappa)+\kappa\sqrt{1-\frac{4\gamma\alpha}{1+2\gamma\alpha+\gamma^2\ell^2}} <1.
\end{equation}
\end{corollary}
\begin{proof}
Since $\lambda =\mu =2$, $\gamma =\delta$, and $\alpha >0$, one can check that \eqref{e:parameters} holds with $\beta =0$. Now apply Theorem~\ref{t:linear} and note that $\varepsilon =\varepsilon_1 =0$ in this case.
\end{proof}

Next, we present another case of the classical DR algorithm when $A$ is monotone and $B$ is strongly monotone. We note the \emph{exchange} of monotonicity assumptions on $A$ and $B$ in Corollaries~\ref{c:classicalDR_gis} and \ref{c:classicalDR}, and that in the latter result, we consider only the $\kappa=1/2$ case for simplicity.

\begin{corollary}[linear convergence of classical DR algorithm]
\label{c:classicalDR}
Suppose that $A$ is monotone and Lipschitz continuous with constant $\ell$, that $B$ is maximally $\beta$-monotone with $\beta\in \RPP$, and that $\zer(A+B)\neq \varnothing$. Suppose also that $\lambda =\mu =2$, $\kappa=1/2$, and $\gamma =\delta\in \RPP$. 
Then $T$ is Lipschitz continuous with constant
\begin{equation}\label{e:classicalDR2}
\rho:=\frac{1}{2(1+\gamma\beta)}\left(
\sqrt{(1+2\gamma\beta)^2-4\gamma\beta(1+\gamma\beta)\min\Big\{\frac{1}{1+\gamma\ell},\frac{1}{1+\gamma^2\ell^2}\Big\}}+1\right) <1.
\end{equation}
Furthermore, if the monotonicity assumption of $A$ is replaced by  
\begin{equation}
\label{e:Aequal}
\forall x,y\in \dom A,\quad \scal{x-y}{Ax-Ay} =0,
\end{equation}
then the Lipschitz constant of $T$ is improved to
\begin{equation}
\label{e:sharp}
\overline{\rho}:=\frac{1}{2(1+\gamma\beta)}\left(
\sqrt{(1+2\gamma\beta)^2-\frac{4\gamma\beta(1+\gamma\beta)}{1+\gamma^2\ell^2}}+1\right) <1.
\end{equation}
\end{corollary}
\begin{proof}
Since $\lambda =\mu =2$, $\gamma =\delta>0$, and $\beta >0$, it is clear that \eqref{e:parameters} is satisfied with $\alpha =0$. 
Applying Theorem~\ref{t:linear},
we obtain that $T$ is Lipschitz continuous with constant
\begin{equation}\label{e:classicalDR1}
\rho:=\frac{1}{2}\left(\sqrt{(1+\varepsilon)^2-4\varepsilon\alpha_J}+1-\varepsilon\right),
\end{equation}
where $\varepsilon :=\frac{\gamma\beta}{1+\gamma\beta}$. Then
\begin{subequations}
\begin{align}
\rho 
&=\frac{1}{2}\left(\sqrt{\Big(\frac{1+2\gamma\beta}{1+\gamma\beta}\Big)^2-\frac{4\gamma\beta}{1+\gamma\beta}\alpha_J}+\frac{1}{1+\gamma\beta}\right)\\
&=\frac{1}{2(1+\gamma\beta)}\left(\sqrt{(1+2\gamma\beta)^2-4\gamma\beta(1+\gamma\beta)\alpha_J}+1\right).
\end{align}
\end{subequations}
Now, it follows from \eqref{e:alpha_J} that
\begin{subequations}
\begin{align}
\alpha_J 
&=\begin{cases}
\dfrac{1}{1+\gamma^2\ell^2} &\text{if~} \gamma\ell\geq 1,\\[10pt]
\dfrac{1}{1+\gamma\ell}&\text{if~} \gamma\ell\leq 1
\end{cases}\\
&=\min\left\{\frac{1}{1+\gamma\ell},\frac{1}{1+\gamma^2\ell^2}\right\},
\end{align}
\end{subequations}
which yields \eqref{e:classicalDR2}.

Finally, if $A$ satisfies \eqref{e:Aequal}, then, again by \eqref{e:alpha_J},
\begin{equation}
\alpha_J =\frac{1}{1+\gamma^2\ell^2},
\end{equation}
and we get \eqref{e:sharp}.
\end{proof}

\begin{remark}[improved Lipschitz constant for classical DR operator]
\label{r:L_improved}
For the classical DR operator ($\lambda=\mu=2$ and $\kappa=1/2$), the Lipschitz constant obtained in Corollary~\ref{c:classicalDR} is sharper than the one obtained in \cite[Theorem~4.4(i)]{MV18}. Indeed, by setting $\gamma=\delta=1$, the Lipschitz constant of \cite[Theorem~4.4(i)]{MV18} is
\begin{subequations}
\begin{align}
r&=\frac{1}{2(1+\beta)}\left(\sqrt{2\beta^2+2\beta+1
+2\left(1-\frac{1}{(1+\ell)^2}-\frac{1}{1+\ell^2}\right)\beta(1+\beta)}+1\right)\\
&=\frac{1}{2(1+\beta)}\left(\sqrt{(1+2\beta)^2
-2\beta(1+\beta)\left(\frac{1}{(1+\ell)^2}+\frac{1}{1+\ell^2}\right)}+1\right),
\end{align}
\end{subequations}
while Corollary~\ref{c:classicalDR} gives the Lipschitz constant
\begin{equation}
\rho=\frac{1}{2(1+\beta)}\left(\sqrt{(1+2\beta)^2
-2\beta(1+\beta)\min\left\{\frac{2}{1+\ell},\frac{2}{1+\ell^2}\right\}}+1\right).
\end{equation}
One can check that
\begin{equation}
\min\left\{\frac{2}{1+\ell},\frac{2}{1+\ell^2}\right\}
>\frac{1}{(1+\ell)^2}+\frac{1}{1+\ell^2}.
\end{equation}
Therefore, $\rho$ is strictly less than $r$.

Regarding the second part of Corollary~\ref{c:classicalDR}, we note that Remark~\ref{r:equality} provides a class of operators satisfying \eqref{e:Aequal} and that the Lipschitz constant \eqref{e:sharp} was shown to be sharp in \cite{MV18}.
\end{remark}

\begin{remark}[choosing the parameter $\gamma$ for best Lipschitz constant]
When the Lipschitz constant $\ell$ of $A$ and the monotonicity constant $\beta$ of $B$ are known, in order to find the best Lipschitz constant for the classical DR operator, one can sketch $\rho$ in \eqref{e:classicalDR2} as a function of $\gamma$ and approximate numerically the value $\gamma$ that yields the minimum of $\rho$. It is, however, not clear how to obtain an explicit formula for the best such value. Indeed, a similar situation was also mentioned in \cite[Remark~5.4]{MV18}.
\end{remark}

As a counterpart of Theorem~\ref{t:linear}, we next consider the adaptive DR algorithm for the case in which $B$ is Lipschitz continuous. For this case, however, we need an additional assumption that $B$ is a linear operator, which implies that $J_2$ and $R_2$ are also linear. To make the argument more symmetric, we will prove an equivalent form of \eqref{e:parameters}.

\begin{lemma}
\label{l:equivalent}
Suppose that $(\gamma,\delta,\lambda,\mu)\in \RPP^2\times\left]1,+\infty\right[^2$. Then
\eqref{e:parameters} is equivalent to
\begin{subequations}
\label{e:parameters'}
\begin{align}
&1+2\delta\beta >0,\label{e:parameters'_a}\\
&\lambda\in \left[2-2\delta\alpha, 2+2\delta\beta\right], \label{e:parameters'_b}\\
&(\mu-1)(\lambda-1) =1, \text{~and~} \gamma =(\mu-1)\delta. \label{e:parameters'_c}
\end{align}
\end{subequations}
\end{lemma}
\begin{proof}
It suffices to prove one implication that \eqref{e:parameters} implies \eqref{e:parameters'} because the converse is totally similar.
First, it is clear that \eqref{e:parameters'_c} is equivalent to \eqref{e:parameters_c}. For the reminder of the proof, we will use $\lambda >1$, $(\lambda-1)(\mu-1) =1$, and $\delta =(\lambda-1)\gamma$. By \eqref{e:parameters_b},  
\begin{subequations}
\begin{align}
\lambda-(2-2\delta\alpha) &=\lambda-1-(\lambda-1)(\mu-1)+2(\lambda-1)\gamma\alpha 
=(\lambda-1)(2+2\gamma\alpha-\mu)\geq 0,\\
2+2\delta\beta-\lambda &=(\lambda-1)(\mu-1)+2(\lambda-1)\gamma\beta+1-\lambda =(\lambda-1)\big(\mu-(2-2\gamma\beta)\big)\geq 0.
\end{align}
\end{subequations}
Thus, $1+2\delta\beta =(2+2\delta\beta-\lambda)+(\lambda-1) >0$, which completes \eqref{e:parameters'}.
\end{proof}

\begin{theorem}[linear convergence when $B$ is Lipschitz]
\label{t:linearB}
Suppose that either
\begin{enumerate}
\item\label{t:linearB_full} 
$A$ is maximally $\alpha$-monotone, $B$ is linear, $\beta$-monotone, and Lipschitz continuous with constant $\ell$, and $\alpha+\beta >0$; or
\item\label{t:linearB_reduce} 
$A$ is maximally $\alpha$-monotone, $B$ is linear and Lipschitz continuous with constant $\ell <\alpha$, and $\beta :=-\ell$.
\end{enumerate}
Suppose also that $\zer(A+B)\neq \varnothing$ and that $(\gamma,\delta,\lambda,\mu)\in\RPP^2\times \left]1,+\infty\right[^2$ satisfies \eqref{e:parameters}.
Then $T$ is Lipschitz continuous with constant
\begin{equation}
\rho :=(1-\kappa)
\sqrt{\big(1+\varepsilon_2(\mu-1)\big)^2-\varphi\alpha_J} +\kappa(1-\varepsilon(\mu-1))\sqrt{1-\frac{\lambda(2+2\delta\beta-\lambda)}{1+2\delta\beta+\delta^2\ell^2}} <1,
\end{equation} 
where
\begin{subequations}
\begin{align}
&\varepsilon :=\frac{\lambda-(2-2\delta\alpha)}{2(1+\gamma\alpha)},\ \varepsilon_2:=\frac{\kappa\varepsilon}{1-\kappa},\\
&\varphi :=\varepsilon_1\mu[2(1+\delta\beta)+\varepsilon_1\big(\mu(1+2\delta\beta)-2(1+\delta\beta)\big)],\\
&\text{$\alpha_J$ as in \eqref{e:alpha_J} with $(\alpha, \gamma)$ replaced by $(\beta, \delta)$}.
\end{align}
\end{subequations}
Consequently, every DR sequence $(x_n)_\nnn$ generated by $T$ converges strongly to the unique fixed point $\overline{x}$ of $T$ with linear rate $\rho$.
\end{theorem}
\begin{proof}
For the same reason as in the proof of Theorem~\ref{t:linear}, we only prove the result under assumption \ref{t:linearB_full}.
Notice that $B$ is maximally $\alpha$-monotone due to Proposition~\ref{p:maxalpha}\ref{p:maxalpha_cont}. Now, by Lemma~\ref{l:domT}, 
\begin{equation}
\min\{1+\gamma\alpha,1+\delta\beta\} >0,
\end{equation}
and all operators $J_1,J_2$ and $T$ are single-valued and have full domain.

Since $B$ is linear, so are $J_2 =J_{\delta B} =(\Id+\delta B)^{-1}$ and $R_2 =(1-\mu)\Id+\mu J_2$. We can thus write
\begin{subequations}
\begin{align}
T 
&=(1-\kappa)\Id-(1-\kappa)\varepsilon_2R_2 +\kappa R_2R_1+\kappa\varepsilon R_2\\
&=(1-\kappa)Q_2 +\kappa R_2Q_1,
\end{align}
\end{subequations}  
where $Q_1 :=R_1+\varepsilon\Id$ and $Q_2 :=\Id-\varepsilon_2R_2$.

Now, by Lemma~\ref{l:equivalent}, \eqref{e:parameters} is equivalent to \eqref{e:parameters'}. Proceeding similarly to the proof of Theorem~\ref{t:linear}, we derive that $Q_2$ is Lipschitz continuous with constant 
\begin{equation}
\sqrt{(1+\varepsilon_2(\mu-1))^2 -\varphi\alpha_J}\leq 1+\varepsilon_2(\mu-1),
\end{equation}  
that $R_2Q_1$ is Lipschitz continuous with constant
\begin{equation}
(1-\varepsilon(\mu-1))\sqrt{1-\frac{\lambda(2+2\delta\beta-\lambda)}{1+2\delta\beta+\delta^2\ell^2}}
\leq 1-\varepsilon(\mu-1),
\end{equation} 
and that at least one of these two inequalities is strict. The conclusion thus follows.
\end{proof}

\begin{remark}
It is worth pointing out that the sum of $\alpha$- and $\beta$-monotone operators with $\alpha+\beta\geq 0$ can be transformed into the sum of two monotone operators by shifting the identity between them as 
\begin{equation}\label{e:shift}
A+B = \Big(A+\frac{\beta-\alpha}{2}\Id\Big)+\Big(B+\frac{\alpha-\beta}{2}\Id\Big)
=:\widetilde{A} + \widetilde{B}.
\end{equation}
Then one can apply the classical DR algorithm for two \emph{new} monotone operators $\widetilde{A}$ and $\widetilde{B}$. However, this is the algorithm that operates on {\em different} operators. Here, our main goal is to show the behavior of the DR algorithm on original data and the smooth transition from the classical case to the adaptive case of the DR algorithm. This approach might be especially helpful when the resolvents are given as \emph{black boxes}, in which case one just needs to adjust the algorithm using corresponding parameters.

When involving two shifted operators like $\widetilde{A}$ and $\widetilde{B}$, it is natural to seek a shifting strategy to obtain the optimal linear convergence rate in Theorem~\ref{t:linear} or Theorem~\ref{t:linearB}. The answer is not clear to us as we hope to address the issue in some future work.
\end{remark}

\section{Applications to structured minimization problems}
\label{s:applications}

Given a function $f\colon X\to \left]-\infty,+\infty\right]$, we recall that $f$ is \emph{proper} if 
\begin{equation}
\dom f :=\menge{x\in X}{f(x) <+\infty}\neq \varnothing
\end{equation}
and \emph{lower semicontinuous} if 
\begin{equation}
\forall x\in \dom f,\quad f(x)\leq \liminf_{z\to x} f(z).
\end{equation}
The function $f$ is said to be \emph{$\alpha$-convex} (see, e.g., \cite[Definition~4.1]{Vial83}) for some $\alpha\in \RR$ if $\forall x,y\in\dom f$, $\forall\kappa\in \left]0,1\right[$,
\begin{equation}
\label{e:alpha-cvx}
f((1-\kappa) x+\kappa y) +\frac{\alpha}{2}\kappa(1-\kappa)\|x-y\|^2\leq (1-\kappa)f(x)+\kappa f(y).
\end{equation}
We say that $f$ is \emph{convex} if $\alpha =0$, \emph{strongly convex} if $\alpha >0$, and \emph{weakly convex} if $\alpha <0$. It is worthwhile noting that \eqref{e:alpha-cvx} is equivalent to
\begin{equation}
f((1-\kappa) x+\kappa y) -\frac{\alpha}{2}\|(1-\kappa) x+\kappa y\|^2\leq (1-\kappa)\Big(f(x) -\frac{\alpha}{2}\|x\|^2\Big) +\kappa \Big(f(y) -\frac{\alpha}{2}\|y\|^2\Big)
\end{equation}
due to \eqref{e:identity}. Thus, 
\begin{equation}\label{e:alphcvx-cvx}
\text{$f$ is $\alpha$-convex}\ \iff\  \text{$f-\frac{\alpha}{2}\|\cdot\|^2$ is convex}.
\end{equation}

In this section, we focus on an important application of the adaptive DR algorithm to the $(\alpha,\beta)$-convex minimization problem, which can be stated as
\begin{equation}
\label{e:minsum}
\min_{x\in X} \{f(x)+g(x)\}
\end{equation}
where $f$ and $g$ are respectively $\alpha$- and $\beta$-convex functions.
To formulate the adaptive DR algorithm for \eqref{e:minsum}, we also recall that the \emph{proximity operator} of a proper function $f\colon X\to \RX$ with parameter $\gamma\in \RPP$ is the mapping $\prox_{\gamma f}\colon X\rightrightarrows X$ defined by
\begin{equation}
\label{e:prox}
\forall x\in X, \quad \prox_{\gamma f}(x) :=\argmin_{z\in X}\left(f(z)+\frac{1}{2\gamma}\|z-x\|^2\right).
\end{equation}
Now let $(\gamma,\delta,\lambda,\mu)\in \RPP^4$ and $\kappa\in \left]0,1\right[$. The adaptive DR algorithm for \eqref{e:minsum} is given by
\begin{equation}
\forall\nnn,\quad
x_{n+1}\in Tx_n, 
\end{equation}
where
\begin{subequations}\label{e:T_prox}
\begin{align}
T &:=(1-\kappa)\Id+\kappa R_2R_1,\\
R_1&:=(1-\lambda)\Id+\lambda\prox_{\gamma f},\\
R_2&:=(1-\mu)\Id+\mu\prox_{\delta g}.
\end{align}
\end{subequations}
Next, we will collect necessary concepts from convex analysis and establish that the adaptive DR operators in \eqref{e:T_prox} is indeed a special case of \eqref{e:JRT} when applied to subdifferential operators. In particular, we will show in Lemma~\ref{l:prox} that for $\alpha$-convex functions, proximity operators are exactly resolvents of \emph{Fr\'echet subdifferentials}. We note that this connection is well known for convex functions (see, e.g., \cite[Proposition~16.44]{BC17}), where the Fr\'echet subdifferential reduces to the \emph{classical convex subdifferential}.

Recall that the Fr\'echet subdifferential of $f$ at $x$ is defined by 
\begin{equation}
\label{e:Frechet}
\widehat{\partial}f(x) :=\Menge{u\in X}{\liminf_{z\to x}\frac{f(z)-f(x)-\scal{u}{z-x}}{\|z-x\|}\geq 0}.
\end{equation}
It is known that if $f$ is differentiable at $x$, then $\widehat{\partial}f(x) =\{\nabla f(x)\}$. When $f$ is a proper convex function, the Fr\'echet subdifferential coincides with the classical convex subdifferential (see, e.g., \cite[Theorem~1.93]{Mor06}), i.e.,
\begin{equation}
\widehat{\partial}f(x) =\partial f(x) :=\menge{u\in X}{\forall z\in X,\ f(z)-f(x)\geq \scal{u}{z-x}}.
\end{equation}

\begin{fact}[subdifferential sum rule]
\label{f:sumrule}
Let $f\colon X\to \left]-\infty,+\infty\right]$ be proper and $\varphi\colon X\to \left]-\infty,+\infty\right]$ be differentiable at $x\in\dom f$. Then
\begin{equation}
\widehat{\partial}(f+\varphi)(x)=\widehat{\partial}f(x)+\nabla\varphi(x).
\end{equation}
\end{fact}
\begin{proof}
This follows from \cite[Proposition~1.107(i)]{Mor06}.
\end{proof}

\begin{lemma}[proximity operators of $\alpha$-convex functions]
\label{l:prox}
Let $f\colon X\to \left]-\infty,+\infty\right]$ be proper, lower semicontinuous, and $\alpha$-convex. Also let $\gamma\in \RPP$ be such that $1+\gamma\alpha >0$. Then the following hold:
\begin{enumerate}
\item\label{l:prox_diff} 
$\widehat{\partial}f$ is maximally $\alpha$-monotone.
\item\label{l:prox_equal} 
$\prox_{\gamma f} =J_{\gamma\widehat{\partial}f}$ is single-valued and has full domain.
\end{enumerate}
\end{lemma}
\begin{proof}
According to \eqref{e:alphcvx-cvx}, the function $h :=f-\frac{\alpha}{2}\|\cdot\|^2$ is convex.

\ref{l:prox_diff}: By Fact~\ref{f:sumrule},
\begin{equation}
\widehat{\partial}f =\widehat{\partial}\big(h+\frac{\alpha}{2}\|\cdot\|^2\big) =\widehat{\partial}h+\alpha\Id.
\end{equation}
Since $h$ is proper lower semicontinuous convex, we learn from \cite[Theorem~21.2]{BC17} that $\widehat{\partial}h$ is maximally monotone, which implies that $\widehat{\partial}f$ is maximally $\alpha$-monotone due to Lemma~\ref{l:transfer}\ref{l:transfer_maxmono}.

\ref{l:prox_equal}: By \ref{l:prox_diff} and Proposition~\ref{p:resol}, $J_{\gamma\widehat{\partial}f}$ is single-valued and has full domain. Let $x\in X$ and set $\varphi :=f +\frac{1}{2\gamma}\|\cdot-x\|^2$. Then
\begin{subequations}
\begin{align}
\varphi(z) &=f(z) +\frac{1}{2\gamma}\|z-x\|^2\\ 
&=\Big(f(z)-\frac{\alpha}{2}\|z\|^2 \Big) +\frac{1+\gamma\alpha}{2\gamma}\left\|z-\frac{1}{1+\gamma\alpha}x\right\|^2 -\frac{\alpha}{2(1+\gamma\alpha)}\|x\|^2.
\end{align}
\end{subequations}
Since $h=f-\frac{\alpha}{2}\|\cdot\|^2$ is convex, so is $\varphi$. Using \eqref{e:prox} and Fact~\ref{f:sumrule}, we have
\begin{subequations}
\begin{align}
p\in \prox_{\gamma f}(x) &\iff 0\in \partial\varphi(p) =\widehat{\partial}f(p) +\frac{1}{\gamma}(p-x)\\
&\iff x\in (\Id+\gamma\widehat{\partial}f)(p)\\
&\iff p\in J_{\gamma\widehat{\partial}f}(x),
\end{align}
\end{subequations}
and the conclusion follows.
\end{proof}

\begin{lemma}
\label{l:chainrule}
Let $f\colon X\to \left]-\infty,+\infty\right]$ and $g\colon X\to \left]-\infty,+\infty\right]$ be respectively $\alpha$- and $\beta$-convex. Then  $f+g$ is $(\alpha+\beta)$-convex. Moreover, if additionally $\alpha+\beta\geq 0$, then $\zer(\widehat{\partial}f+\widehat{\partial}g)\subseteq \argmin(f+g)$.
\end{lemma}
\begin{proof}
We write
\begin{equation}
f+g=\Big(f-\frac{\alpha}{2}\|\cdot\|^2\Big)
+\Big(g-\frac{\beta}{2}\|\cdot\|^2\Big)
+\frac{\alpha+\beta}{2}\|\cdot\|^2,
\end{equation}
which together with \eqref{e:alphcvx-cvx} implies the $(\alpha+\beta)$-convexity of $f+g$. Next, let $x\in\zer(\widehat\partial f+\widehat\partial g)$. If $\alpha+\beta\geq 0$, then $f+g$ is convex, and we have that
\begin{equation}
0\in \widehat{\partial}f(x)+\widehat{\partial}g(x)\subseteq \widehat{\partial}(f+g)(x) =\partial(f+g)(x),
\end{equation}
so $x\in \argmin(f+g)$. The proof is complete.
\end{proof}

\begin{theorem}[adaptive DR algorithm for $(\alpha,\beta)$-convex minimization]
\label{t:f-cvg}
Let $f\colon X\to\left]-\infty,+\infty\right]$ 
and $g\colon X\to \left]-\infty,+\infty\right]$ 
be proper and lower semicontinuous. Suppose also that $f$ and $g$ are respectively $\alpha$- and $\beta$-convex with $\zer(\widehat\partial f+\widehat\partial g)\neq \varnothing$, and that one of the following holds:
\begin{enumerate}
\item\label{t:f-cvg_adapt}
(Adaptive DR algorithm)
$\alpha+\beta\geq 0$ and $(\gamma,\delta,\lambda,\mu)\in \RPP^2\times \left]1,+\infty\right[^2$ satisfies \eqref{e:parameters}.
\item\label{t:f-cvg_DR} 
(Classical DR algorithm) $\lambda =\mu =2$, $\gamma=\delta\in \RPP$, and 
\begin{subequations}
\begin{alignat}{2}
&\text{either}\quad &&\alpha =\beta =0;\\
&\text{or}\quad &&\alpha+\beta >0 \text{~and~} 1+\gamma\frac{\alpha\beta}{\alpha+\beta} >\kappa. 
\end{alignat}
\end{subequations}
\end{enumerate} 
Then every DR sequence $(x_n)_\nnn$ generated by $T$ converges weakly to a point $\overline{x}\in \Fix T$ with $\prox_{\gamma f}(\overline{x})\in \zer(\widehat{\partial}f+\widehat{\partial}g)\subseteq \argmin(f+g)$ and the rate of asymptotic regularity of $T$ is $o(1/\sqrt{n})$. Moreover, if $\alpha+\beta >0$, then $(\prox_{\gamma f}(x_n))_\nnn$ and $(\prox_{\delta g}(R_1 x_n))_\nnn$ converge strongly to $\prox_{\gamma f}(\overline{x})$ and $\argmin(f+g) =\{\prox_{\gamma f}(\overline{x})\}$.
\end{theorem}
\begin{proof}
In view of Lemmas~\ref{l:prox} and \ref{l:chainrule},
we apply Theorem~\ref{t:DRcvg} to $A =\widehat{\partial}f$ and $B =\widehat{\partial}g$.
\end{proof}

\begin{remark}[strongly and weakly convex minimization]
In \cite[Theorems~4.4 and 4.6]{GHY17}, the authors proved the convergence of the classical DR algorithm for problem \eqref{e:minsum} when $f$ and $g$ are respectively $\alpha$- and $\beta$-convex functions in a Euclidean space with either $\alpha >-\beta >0$ or ${\beta >-\alpha >0}$. Roughly speaking, these results require that the strong convexity {\em strictly outweighs} the weak counterpart. 

In contrast, our approach (Theorem~\ref{t:f-cvg}) for this problem assumes $\alpha+\beta\geq 0$, which means the weak convexity only needs to be {\em neutralized}. Under this assumption, we {\em adapt} the parameters so that the convergence is guaranteed. Let us recall that when both functions in \eqref{e:minsum} are convex, we may just assume there is neither a strong nor a weak component, i.e., $\alpha=\beta=0$, and obtain the convergence for the classical DR algorithm.

Recently, for the $\alpha+\beta=0$ case, the classical DR algorithm has been considered in \cite{GH18}, where the convergence requires that one function is strongly convex with Lipschitz continuous gradient. We note that in this case, the convergence of the adaptive DR algorithm is established in Theorem~\ref{t:f-cvg}\ref{t:f-cvg_adapt} without any differentiability assumption on the functions.
\end{remark}

Finally, we present a linear convergence result under Lipschitz assumption on the gradient of $f$. For other linear convergence results of related splitting methods in the context of structured minimization problems, we refer interested readers to \cite{DavY17,DenY16} and the references therein.

\begin{theorem}[linear convergence when $\nabla f$ is Lipschitz continuous]
Let $f\colon X\to \RR$ be a differentiable function whose gradient $\nabla f$ is Lipschitz continuous with constant $\ell$, and let $g\colon X\to \RX$ be a proper lower semicontinuous function. Suppose that either
\begin{enumerate}
\item $f$ is $\alpha$-convex,
$g$ is $\beta$-convex, and $\alpha+\beta>0$; or
\item $g$ is $\beta$-convex with $\beta>\ell$, and $\alpha:=-\ell$.
\end{enumerate}
Suppose also that $\zer(\nabla f+\widehat{\partial} g)\neq\varnothing$ and that $(\gamma,\delta,\lambda,\mu)\in\RPP^2\times\left]1,+\infty\right[^2$ satisfies \eqref{e:parameters}. Then the adaptive DR operator $T$ is Lipschitz continuous with constant less than $1$. Consequently, every DR sequence $(x_n)_\nnn$ generated by $T$ converges strongly to the unique fixed point $\overline{x}$ of $T$ with linear rate.
\end{theorem}
\begin{proof}
Apply Theorem~\ref{t:linear} with $A=\widehat{\partial}f=\nabla f$ and $B=\widehat{\partial} g$.
\end{proof}

\section{Conclusion}
\label{s:concl}

We have studied the adaptive DR algorithm for finding a zero of the sum of $\alpha$- and $\beta$-monotone operators. The adaptive parameters provide great flexibility for adjusting the DR algorithm so that the convergence is guaranteed. We have derived the rate of asymptotic regularity $o(1/\sqrt{n})$ for the adaptive DR operator. When the strong convexity strictly outweighs the weak one, we have further obtained the strong convergence of shadow sequences to the solution of the original problem. Global linear convergence is also achieved with a sharp rate in several important cases. Our new approach, on the one hand, generalizes previous works in the same direction and, on the other hand, unifies the convergence analysis of the DR algorithm under monotone-type assumptions.

\section*{Acknowledgement}
The authors are grateful to the associate editor and the two anonymous referees for their constructive comments and suggestions.
MND was partially supported by the Australian Research Council (ARC) Discovery Project DP160101537 and by the Priority Research Centre for Computer-Assisted Research Mathematics and its Applications (CARMA) at the University of Newcastle. HMP was partially supported by Autodesk, Inc via a gift made to the Department of Mathematical Sciences, University of Massachusetts Lowell.

\end{document}